\documentclass[11pt,reqno]{amsart}

\usepackage{amsmath,amssymb,amsthm}
\usepackage[numbers]{natbib}
\usepackage[dvips]{graphicx}

\theoremstyle{plain}
\newtheorem{theorem}{Theorem}[section]	
\newtheorem{lemma}{Lemma}[section]
\newtheorem{corollary}{Corollary}[section]
\newtheorem{proposition}{Proposition}[section]
\theoremstyle{definition}
\newtheorem{definition}{Definition}[section]
\newtheorem{remark}{Remark}[section]
\newtheorem{example}{Example}[section]

\DeclareMathOperator{\Card}{Card}

\DeclareMathOperator{\Var}{Var}
\DeclareMathOperator{\re}{re}
\DeclareMathOperator{\si}{si}
\DeclareMathOperator{\spa}{sp}

\newcommand{\R}{\mathbb{R}}

\newcommand{\mS}{\mathcal{S}}

\newcommand{\mA}{\mathcal{A}}
\renewcommand{\tilde}{\widetilde}
\renewcommand{\hat}{\widehat}
\renewcommand{\qed}{\hfill{\tiny \ensuremath{\blacksquare} }}

\newcommand{\Ep}{{\mathrm{E}}}

\newcommand{\En}{{\mathbb{E}_n}}
\renewcommand{\Pr}{{\mathrm{P}}}

\begin{document}

\title[CLT and Bootstrap in High Dimensions]{Central Limit Theorems and Bootstrap in High Dimensions}
%\thanks{V. Chernozhukov and D. Chetverikov  are supported by a National Science Foundation grant. K.Kato is supported by the Grant-in-Aid for Young Scientists (B) (25780152), the Japan Society for the Promotion of Science. }

\author[Chernozhukov]{Victor Chernozhukov}
\author[Chetverikov]{Denis Chetverikov}
\author[Kato]{Kengo Kato}

\address[V. Chernozhukov]{
Department of Economics and Center for Statistics, MIT, 50 Memorial Drive, Cambridge, MA 02142, USA.}
\email{vchern@mit.edu}

\address[D. Chetverikov]{
Department of Economics, UCLA, Bunche Hall, 8283, 315 Portola Plaza, Los Angeles, CA 90095, USA.}
\email{chetverikov@econ.ucla.edu}

\address[K. Kato]{
Graduate School of Economics, University of Tokyo, 7-3-1 Hongo Bunkyo-ku, Tokyo 113-0033, Japan.}
\email{kkato@e.u-tokyo.ac.jp}

\date{First version: February 07, 2014. This version: \today.}

\begin{abstract}
This paper derives central limit and bootstrap theorems for probabilities that sums of centered high-dimensional random vectors hit hyperrectangles and sparsely convex sets. Specifically, we derive Gaussian and bootstrap approximations for probabilities $\Pr(n^{-1/2}\sum_{i=1}^n X_i\in A)$ where $X_1,\dots,X_n$ are independent random vectors in $\R^p$ and $A$ is a hyperrectangle, or, more generally, a sparsely convex set, and show that the approximation error converges to zero even if $p=p_n\to \infty$ as $n \to \infty$ and $p \gg n$; in particular, $p$ can be as large as $O(e^{Cn^c})$ for some constants $c,C>0$. The result holds uniformly over all hyperrectangles, or more generally, sparsely convex sets, and does not require any restriction on the correlation structure among coordinates of $X_i$.  Sparsely convex sets are sets that can be represented as intersections of many convex sets whose indicator functions depend  only on a small subset of their arguments, with hyperrectangles being a special case. 
\end{abstract}

%\subjclass[2000]{60F17, 62E17, 62G20}

\thanks{We are grateful to Evarist Gin\'{e}, Friedrich G\"{o}tze, Ramon van Handel, Vladimir Koltchinskii, Richard Nickl,  and Larry Wasserman, Galyna Livshyts, and Karim Lounici for useful discussions.}

\maketitle

\section{Introduction}
Let $X_1,\dots,X_n$ be independent random vectors in $\R^p$ where $p \geq 3$ may be large or even much larger than $n$.  Denote by $X_{ij}$ the $j$-th coordinate of $X_{i}$, so that $X_i=(X_{i1},\dots,X_{i p})'$.
We assume that each $X_i$ is centered, namely $\Ep[X_{i j}]=0$, and $\Ep [X_{ij}^{2}] < \infty$ for all $i=1,\dots,n$ and $j=1,\dots,p$. 
Define the normalized sum
\[
S^X_n:=(S^X_{n1},\dots,S^X_{np})':=\frac{1}{\sqrt{n}}\sum_{i=1}^n X_i.
\]
We consider Gaussian approximation to $S_{n}^{X}$, and to this end, let $Y_1,\dots, Y_n$ be independent centered Gaussian random vectors in $\R^p$ such that each $Y_i$ has the same covariance matrix as $X_i$, that is, $Y_i\sim N(0,\Ep[X_i X_i'])$. 
Define the normalized sum for the Gaussian random vectors:
\[
S^Y_n:=(S^Y_{n1},\dots,S^Y_{np})':=\frac{1}{\sqrt{n}}\sum_{i=1}^n Y_i.
\]
We are interested in bounding the quantity
\begin{equation}
\label{eq: rho definition general}
\rho_n(\mA):=\sup_{A \in \mA}|\Pr(S^X_n\in A)-\Pr(S^Y_n\in A)|,
\end{equation}
where  $\mA$ is a class of Borel sets in $\R^p$. 

Bounding $\rho_{n}(\mA)$ for various classes $\mA$ of sets in $\R^{p}$, with a special emphasis on explicit dependence on the dimension $p$ in the bounds, has been studied by a number of authors; see, for example, \cite{B86}, \cite{Bentkus03}, \cite{B75}, \cite{G91}, \cite{N76}, \cite{S68}, \cite{S81}, \cite{S80}, and \cite{S77}; we refer to \cite{CCK12d} for an exhaustive literature review.
Typically, we are interested in how fast $p = p_{n} \to \infty$ is allowed to grow while guaranteeing $\rho_n(\mA)\to 0$. 
In particular, Bentkus \cite{Bentkus03} established one of the sharpest results in this direction which states that when $X_{1},\dots,X_{n}$ are i.i.d. with $\Ep[X_{i} X_{i}'] = I$ ($I$ denotes the $p \times p$ identity matrix), 
\begin{equation}
\label{eq: Bentkus bound}
\rho_n(\mA) \leq  C_{p}(\mA)\frac{\Ep[\|X_{1}\|^3]}{\sqrt{n}}, 
\end{equation}
where $C_{p}(\mA)$ is a constant that depends only on $p$ and $\mA$; for example, $C_p(\mA)$ is bounded by a universal constant when $\mA$ is the class of all Euclidean balls in $\R^p$, and $C_p(\mA) \leq 400p^{1/4}$ when $\mA$ is the class of all Borel measurable convex sets in $\R^p$. Note, however, that this bound does not allow $p$ to be larger than $n$ once  we require $\rho_{n}(\mA) \to 0$. Indeed by Jensen's inequality, when $\Ep[X_{1}X_{1}'] = I$, $\Ep[ \|X_{1}\|^{3}] \geq (\Ep[\|X_{1}\|^{2}])^{3/2} = p^{3/2}$, and hence in order to make the right-hand side of (\ref{eq: Bentkus bound}) to be $o(1)$, we at least need $p=o(n^{1/3})$ when $\mA$ is the class of Euclidean balls, and $p = o(n^{2/7})$ when $\mA$ is the class of all Borel measurable convex sets. 
Similar conditions are needed in other papers cited above. It is worthwhile to mention here that, when $\mA$ is the class of all Borel measurable convex sets, it was proved by \cite{N76} that $\rho_n(\mA)\geq c\Ep[\|X_{1}\|^3]/\sqrt{n}$ for some universal constant $c > 0$.

In modern statistical applications, such as high dimensional estimation and multiple hypothesis testing, however, $p$ is often larger or even much larger than $n$. It is therefore interesting to ask whether it is possible to provide a nontrivial class of sets $\mA$ in $\R^p$ for which we would have 
\begin{equation}
\label{eq: general question}
\rho_n(\mA)\to 0 \ \text{{\em even if}} \ p \ \text{{\em is potentially larger or much larger than}} \ n.
\end{equation}

In this paper, we derive bounds on $\rho_{n}(\mA)$ for $\mA=\mA^{\re}$ being the class of all hyperrectangles, or more generally for $\mA\subset\mA^{\si}(a,d)$ being a class of simple convex sets, and show that these bounds lead to results of type (\ref{eq: general question}). We call any convex set a simple convex set if it can be well approximated by a convex polytope whose number of facets is (potentially very large but) not too large; see Section \ref{sec: friendly sets} for details. An extension to simple convex sets is interesting because it allows us to derive similar bounds for $\mA=\mA^{\spa}(s)$ being the class of ($s$-)sparsely convex sets. These are sets that can be represented as an intersection of many convex sets whose indicator functions depend nontrivially at most on $s$ elements of their arguments (for some small $s$).

The sets considered are useful for applications to statistics. In particular, the results for hyperrectangles
and sparsely convex sets are of importance  because they allow us to approximate the distributions of various key statistics that arise in inference for high-dimensional models. For example, the probability that a collection of Kolmogorov-Smirnov type statistics falls below a collection of thresholds
\[
\Pr \left (\max_{j \in J_k}  S_{nj}^X \leq t_k \ \text{for all} \ k=1,\dots,\kappa \right) =  \Pr \left (S_{n}^X   \in A \right)
\]
can be approximated by $\Pr  (S_{n}^Y   \in A )$ within the error margin $\rho_n(\mA^{\re})$; here
 $\{J_k\}$ are (non-intersecting) subsets of $\{1,\dots,p\}$, $\{t_k\}$ are thresholds in the interval $(-\infty, \infty)$,  $ \kappa \geq 1$ is an integer, and $A \in \mA^{\re}$  is a hyperrectangle of the form $\{w \in \R^p: \max_{j \in J_k}  w_j  \leq t_k \ \text{for all} \ k=1,\dots,\kappa\}$.  Another example is the probability that a collection of Pearson type statistics falls below a collection of thresholds
\[
\Pr \Big ( \|(S_{nj}^X)_{j \in J_k}\|^2 \leq t_k \ \text{for all} \ k=1,\dots,\kappa \Big ) =  \Pr \left (S_{n}^X   \in A \right),
\]
which can be approximated by $\Pr(S_{n}^Y  \in A)$ within the error margin $\rho_n(\mA^{\spa}(s))$; 
here  $\{J_k\}$ are subsets of $\{1,\dots,p\}$ of fixed cardinality $s$, $\{t_k\}$  are thresholds in the interval $(0, \infty)$, $\kappa \geq 1$ is an integer, and $A \in \mA^{\spa}(s)$ is a sparsely convex set 
of the form $\{ w\in \R^p: \|(w_j)_{j \in J_k}\|^2  \leq t_k \ \text{for all} \ k=1,\dots,\kappa\}$.   In practice, as we demonstrate, the approximations above
could be estimated using the empirical or multiplier bootstraps.

% (\ref{eq: general question}) holds if $\mA$ is the class of all rectangles in $\R^p$ or, more generally, the class of sparse convex sets, the concept that we introduce below in this paper. %In comparison with our previous result published in \cite{CCK12a}, we obtain a better rate of dependence on $n$.

The results in this paper substantially extend those obtained in \cite{CCK12a} where we considered the class $\mA=\mA^m$ of sets of the form $A=\{w\in \R^p:\max_{j \in J}w_j\leq a\}$ for some $a\in\R$ and $J \subset \{1,\dots,p\}$, but in order to obtain much better dependence on $n$, we employ new techniques. Most notably, as the main ingredient in the new proof, we employ an argument inspired by Bolthausen \cite{Bolthausen84}.  Our paper builds upon our previous work \cite{CCK12a}, which in turn builds on  a number of works listed in the bibliography (see  \cite{CCK12d} for a detailed review and links to the literature).

The organization of this paper is as follows. 
In Section \ref{sec: main results}, we derive a Central Limit Theorem (CLT) for hyperrectangles in high dimensions; that is, we derive a bound on $\rho_n(\mA)$ for $\mA=\mA^{\re}$ being the class of all hyperrectangles and show that the bound converges to zero under certain conditions even when $p$ is potentially larger or much larger than $n$. In Section \ref{sec: friendly sets}, we extend this result by showing that similar bounds apply for $\mA\subset\mA^{\si}(a,d)$ being a class of simple convex sets and for $\mA=\mA^{\spa}(s)$ being the class of all $s$-sparsely convex sets. In Section \ref{sec: bootstrap CLT}, we derive high dimensional empirical and multiplier bootstrap theorems that allow us to approximate $\Pr(S_n^Y\in A)$ for $A\in\mA^{\re}$, $\mA^{\si}(a,d)$, or $\mA^{\spa}(s)$ using the data $X_1,\dots,X_n$. In Section \ref{sec: induction lemma}, we state an important technical lemma, which constitutes the main part of the derivation of our high dimensional CLT. Finally, we provide all the proofs as well as some technical results in the Appendix.

\subsection{Notation} 
For $a \in \R$, $[a]$ denotes the largest integer smaller than or equal to $a$. For $w=(w_1,\dots,w_p)' \in \R^p$ and $y=(y_1,\dots,y_p)' \in \R^p$, we write $w \leq y$ if $w_j \leq y_j$ for all $j=1,\dots,p$. For $y=(y_1,\dots,y_p)'\in \R^p$ and $a \in \R$, we write $y+a = (y_{1}+a,\dots,y_{p} + a)'$. Throughout the paper, $\En[\cdot]$ denotes the average over index $i=1,\dots,n$; that is, it simply abbreviates the notation $n^{-1}\sum_{i=1}^n[\cdot]$. For example, $\En[x_{i j}]=n^{-1}\sum_{i=1}^n x_{i j}$. 
% The notation $a_n\lesssim b_n$ means that there exists a constant $C>0$ such that $a_n\leq C b_n$ for all $n$. 
We also write $X_{1}^{n} := \{ X_1,\dots,X_n \}$. 
For $v\in\R^p$, we use the notation $\|v\|_0:=\sum_{j=1}^p 1\{v_j \neq 0\}$ and $\|v\|=(\sum_{j=1}^p v_j^2)^{1/2}$.
For $\alpha>0$, we define the function $\psi_{\alpha}: [0,\infty) \to [0,\infty)$ by $\psi_{\alpha} (x):=\exp(x^\alpha)-1$, and for a real-valued random variable $\xi$, we define 
\[
\| \xi \|_{\psi_\alpha}:=\inf \{\lambda>0: \Ep[ \psi_{\alpha}( | \xi | /\lambda)] \leq 1\}. 
\]
For $\alpha \in [1,\infty)$, $\|\cdot\|_{\psi_{\alpha}}$ is an Orlicz norm, while for $\alpha \in (0,1)$, $\| \cdot \|_{\psi_{\alpha}}$ is not a norm but a quasi-norm, that is, there exists a constant $K_{\alpha}$ depending only on $\alpha$ such that $\| \xi_{1} + \xi_{2} \|_{\psi_{\alpha}} \leq K_{\alpha} ( \| \xi_{1} \|_{\psi_{\alpha}} + \| \xi_{2} \|_{\psi_{\alpha}})$. Throughout the paper, we assume that $n\geq 4$ and $p\geq 3$.

\section{High Dimensional CLT for Hyperrectangles}
\label{sec: main results}
This section presents a high dimensional CLT for hyperrectangles. 
We begin with presenting an abstract theorem (Theorem \ref{thm: main}); the bound in Theorem \ref{thm: main} is general but depends on the tail properties of the distributions of the coordinates of $X_{i}$ in a nontrivial way. Next we apply this theorem under simple moment conditions and derive more explicit bounds (Proposition \ref{cor: examples}).

Let $\mA^{\re}$ be the class of all hyperrectangles in $\R^p$; that is, $\mA^{\re}$ consists of all sets $A$ of the form
\begin{equation}\label{eq: rectangle}
A=\left\{w\in\R^p:a_j\leq w_j\leq b_j\text{ for all }j=1,\dots,p\right\}
\end{equation}
for some $-\infty\leq a_j\leq b_j\leq\infty$, $j=1,\dots,p$. We will derive a bound on $\rho_n(\mA^{\re})$, and show that under certain conditions it leads to $\rho_n(\mA^{\re})\to 0$ even when $p=p_n$ is potentially larger or much larger than $n$.

To describe the bound, we need to prepare some notation. Define
\[
L_n:=\max_{1\leq j\leq p}\sum_{i=1}^n \Ep[|X_{i j}|^3]/n, 
\]
and for $\phi\geq 1$, define
\begin{equation}\label{eq: M definition}
M_{n,X}(\phi):=n^{-1} \sum_{i=1}^n\Ep\left[\max_{1\leq j\leq p}|X_{ij}|^31\left\{\max_{1\leq j\leq p}|X_{i j}|>\sqrt{n}/(4\phi\log p)\right\}\right].
\end{equation}
Similarly, define $M_{n,Y}(\phi)$ with $X_{ij}$'s replaced by $Y_{ij}$'s in (\ref{eq: M definition}), and let 
\[
M_n(\phi):=M_{n,X}(\phi)+M_{n,Y}(\phi).
\]
The following is the first main result of this paper.
\begin{theorem}[Abstract High Dimensional CLT for Hyperrectangles]\label{thm: main}
Suppose that there exists some constant $b>0$ such that $n^{-1}\sum_{i=1}^n\Ep[X_{i j}^2]\geq b$ for all $j=1,\dots,p$. Then there exist constants $K_1, K_{2} >0$ depending only $b$ such that for every constant $\overline{L}_{n}\geq L_n$, we have
\begin{equation}
\label{eq: main bound}
\rho_n(\mA^{\re})\leq K_1\left[\left(\frac{\overline{L}_{n}^2\log^7 p}{n}\right)^{1/6}+\frac{M_{n}(\phi_n)}{\overline{L}_{n}}\right]
\end{equation}
with 
\begin{equation}\label{eq: psi}
\phi_n:=K_2 \left ( \frac{\overline{L}_{n}^2\log^4 p}{n} \right)^{-1/6}.
\end{equation}
\end{theorem}

\begin{remark}[Key features of Theorem \ref{thm: main}]
(i) The bound (\ref{eq: main bound}) should be contrasted with Bentkus's \cite{Bentkus03} bound (\ref{eq: Bentkus bound}). 
For the sake of exposition, assume that the vectors $X_1,\dots,X_n$ are such that $\Ep[X_{ij}^2]=1$ and for some sequence of constants $B_{n} \geq 1$, $|X_{ij}|\leq B_n$ for all $i=1,\dots,n$ and $j=1,\dots,p$. Then it can be shown that the bound (\ref{eq: main bound}) reduces to
\begin{equation}
\label{eq: our bound comparison}
\rho_n(\mA^{\re}) \leq K \Big( n^{-1} B_n^2\log^7 (p n) \Big)^{1/6}
\end{equation}
for some universal constant $K$; see Proposition \ref{cor: examples} below. Importantly, the right-hand side of (\ref{eq: our bound comparison}) converges to zero even when $p$ is much larger than $n$; indeed we just need $B_n^2\log^7 (p n) = o(n)$ to make $\rho_{n}(\mA^{\re}) \to 0$, and if in addition $B_{n} = O(1)$, the condition reduces to $\log p = o(n^{1/7})$. In contrast, Bentkus's bound (\ref{eq: Bentkus bound}) requires $\sqrt{p} = o(n^{1/7})$ to make $\rho_{n}(\mA) \to 0$ when $\mA$ is the class of all Borel measurable convex sets. 
Hence by restricting the class of sets to the smaller one, $\mA= \mA^{\re}$, we are able to considerably weaken the requirement on $p$, replacing $\sqrt{p}$ by $\log p$.

\medskip

(ii) On the other hand, the bound in (\ref{eq: our bound comparison}) depends on $n$ through $n^{-1/6}$, so that our Theorem \ref{thm: main} does not recover the Berry-Esseen bound when $p$ is fixed. However, given that the rate $n^{-1/6}$ is optimal (in a minimax sense) in CLT in infinite dimensional Banach spaces (see \cite{Bentkus85}),  the factor $n^{-1/6}$  seems nearly optimal in terms of dependence on $n$ in the high-dimensional settings as considered here. 
In addition, examples in \cite{CCK12b} suggest that dependence on $B_n$ is also optimal.  
Hence we conjecture that up to a universal constant,
\[
\Big( n^{-1} B_n^2 (\log p)^{a} \Big)^{1/6}
\] 
for some $a >0 $ is an optimal bound  (in a minimax sense) in the high dimensional setting as considered here.  The value $a =3$ could be motivated by the theory of moderate deviations for self-normalized sums
when all the coordinates of $X_i$ are independent.  \qed
\end{remark}

\begin{remark}[Relation to previous work]
Theorem \ref{thm: main} extends Theorem 2.2 in \cite{CCK12a} where we derived a bound on $\rho_n(\mA^m)$
with  $\mA^m \subset \mA^{\re}$ consisting of all sets of the form
\[
A=\{ w \in \R^p: w_j \leq a \ \text{for all} \ j=1,\dots,p\}
\]
for some $a \in \R$. In particular, we improve the dependence on $n$ from $n^{-1/8}$ in \cite{CCK12a} to $n^{-1/6}$. In addition, we note that extension to the class $\mA^{\re}$ from the class $\mA^m$ is not immediate since in both papers we assume that $\Var (S_{nj}^X)$ is bounded below from zero uniformly in $j=1,\dots,p$, so that it is not possible to directly extend the results in \cite{CCK12a} to the class of hyperrectangles $\mA = \mA^{\re}$ by just rescaling the coordinates in $S_{n}^{X}$. 
\qed
\end{remark}

The bound (\ref{eq: main bound}) depends on $M_{n}(\phi_n)$ whose values are problem specific. Therefore, we now apply Theorem \ref{thm: main} in two specific examples that are most useful in mathematical statistics (as well as other related fields such as econometrics). Let $b,q>0$ be some constants, and let $B_n\geq 1$ be a sequence of constants, possibly growing to infinity as $n \to \infty$. Assume that the following conditions are satisfied:

\begin{itemize}
\item[(M.1)] $n^{-1} \sum_{i=1}^{n}\Ep[X^2_{i j}] \geq b$ {\em for all} $j=1,\dots,p$, 
\item[(M.2)] $n^{-1} \sum_{i=1}^{n}\Ep[|X_{i j}|^{2+k}]\leq B^k_n$ {\em for all} $j=1,\dots,p$ {\em and} $k=1,2$.
\end{itemize}

We consider examples where one of the following conditions holds:

\begin{itemize}
\item[(E.1)] $\Ep[\exp(|X_{i j}|/B_n)] \leq 2$ {\em for all} $i=1,\dots,n$ {\em and} $j=1,\dots,p$,
\item[(E.2)] $\Ep[ (\max_{1\leq j \leq p} |X_{i j}| / B_{n})^q] \leq 2$ {\em for all} $i=1,\dots,n$,
\end{itemize}
In addition, denote
\begin{equation}
\label{eq: definition D}
D_{n}^{(1)} = \left(\frac{B_n^2\log^7 (p n)}{n}\right)^{1/6}, \ D_{n,q}^{(2)} = \left(\frac{B_n^2\log^3 (pn)}{n^{1-2/q}}\right)^{1/3}.
\end{equation}
An application of Theorem \ref{thm: main} under these conditions leads to the following proposition. 
\begin{proposition}[High Dimensional CLT for Hyperrectangles]\label{cor: examples}
Suppose that conditions {\em (M.1)} and {\em (M.2)} are satisfied. Then under {\em (E.1)}, we have
\[
\rho_n(\mA^{\re})
% \leq C \left(\frac{B_n^2\log^7 (p n)}{n}\right)^{1/6}
\leq C D_{n}^{(1)},
\]
where the constant $C$ depends only on $b$; while under {\em (E.2)}, we have
\[
\rho_n(\mA^{\re})
% \leq C\left[\left(\frac{B_n^2\log^7 (p n)}{n}\right)^{1/6}+\left(\frac{B_n^2\log^3 p}{n^{1-2/q}}\right)^{1/3}\right],
\leq C \{ D_{n}^{(1)} + D_{n,q}^{(2)} \},
\]
where the constant $C$ depends only on $b$ and $q$.
\end{proposition}
%\begin{remark}
%Let $c_2,C_2>0$ be some constants. Assume that the condition (E.2) is satisfied. In addition, assume that $p=p_n\leq C_1n^{C_1}$. Then
%$$
%\frac{B_n(\log n)^{3/2}}{n^{1/2-1/q}}\leq C_2n^{-c_2}\text{ implies that }\rho_n\leq Cn^{-c}
%$$
%for some constants $c,C>0$ that depend only on $c_1$, $C_1$, $c_2$, and $C_2$. In addition,
%$$
%\frac{B_n(\log n)^{3/2}}{n^{1/2-1/q}}\to 0\text{ implies that }\rho_n\to 0.
%$$
%Proof of both statements follow directly from Corollary \ref{cor: examples}.
%\end{remark}

\section{High Dimensional CLT for Simple and Sparsely Convex Sets}
\label{sec: friendly sets}

In this section, we extend the results of Section \ref{sec: main results} by considering larger classes of sets; in particular, we consider classes of simple convex sets, and obtain, under certain conditions, bounds that are similar to those in Section \ref{sec: main results} (Proposition  \ref{cor: sparsely convex bodies}). Although an extension to simple convex sets is not difficult, in high dimensional spaces, the class of simple convex sets is rather large. In addition, it allows us to derive similar bounds for classes of sparsely convex sets. These classes in turn may be of interest in statistics where sparse models and techniques have been of canonical importance in the past years.

\subsection{Simple convex sets}

Consider a closed convex set $A\subset\R^p$.  This set can be characterized by its support function:
\[
\mS_{A}: \mathbb{S}^{p-1} \to \R \cup \{\infty\}, \quad v \mapsto  \mS_{A}(v): = \sup\{ w' v: w \in A\},
\]
where $\mathbb{S}^{p-1}:=\{v\in\R^p: \| v \| =1\}$;
in particular,  $A = \cap_{v \in \mathbb{S}^{p-1}} \{ w \in \R^p: w' v \leq \mS_A(v)  \}$.
We say that the set $A$ is {\em $m$-generated} if it is generated by the intersection of $m$ half-spaces (that is, $A$ is a convex polytope with at most $m$ facets). The support  function $\mS_{A}$ of such a set $A$  can be characterized completely by its values   $\{\mS_A(v) : v \in \mathcal{V}(A)\}$ for the set $\mathcal{V}(A)$ consisting of $m$ unit vectors that are outward normal to the facets of $A$. Indeed,
\[
A = \cap_{v \in \mathcal{V}(A)} \{ w \in \R^p: w' v \leq \mS_A(v)  \}.
\]
For $\epsilon>0$ and an $m$-generated convex set $A^m$, we define
\[
A^{m,\epsilon}:=\cap_{v \in \mathcal{V}(A^m)} \{ w \in \R^p: w' v \leq \mS_{A^m}(v) +\epsilon \},
\]
and we say that a convex set $A$ admits an approximation with precision $\epsilon$ by an $m$-generated convex set $A^m$ if
\[
A^m \subset A \subset A^{m,\epsilon}.
\]

Let $a,d>0$ be some constants. Let $\mA^{\si}(a,d)$ be the class of all Borel sets $A$ in $\R^p$ that satisfy the following condition:%We consider classes of convex sets

\begin{itemize}
\item[(C)] \textit{The set $A$ admits an approximation with precision $\epsilon = a/n$ by an $m$-generated convex set $A^m$ where $m \leq  (p n)^d$.}
\end{itemize}

We refer to sets $A$ that satisfy condition (C) as {\em simple convex sets}. 
Note that any hyperrectangle $A \in \mA^{\re}$ satisfies condition (C) with $a=0$ and $d=1$ (recall that $n\geq 4$), and so belongs to the class $\mA^{\si}(0,1)$. For $A\in\mA^{\si}(a,d)$, let $A^m(A)$ denote the corresponding set $A^m$ that appears in condition (C).

We will consider subclasses $\mathcal{A}$ of the class $\mA^{\si}(a,d)$ consisting of sets $A$ such that for $A^m=A^m(A)$ and $\tilde{X}_i=(\tilde{X}_{i1},\dots,\tilde{X}_{im})'=(v'X_i)_{v\in\mathcal{V}(A^m)}$, $i=1,\dots,n$, the following conditions are satisfied:

\begin{itemize}
\item[(M.1$'$)] $n^{-1} \sum_{i=1}^{n}\Ep[\tilde{X}^2_{i j}] \geq b$ {\em for all} $j=1,\dots,m$, 
\item[(M.2$'$)] $n^{-1} \sum_{i=1}^{n}\Ep[|\tilde{X}_{i j}|^{2+k}]\leq B^k_n$ {\em for all} $j=1,\dots,m$ {\em and} $k=1,2$,
\end{itemize}
and, in addition, one of the following conditions is satisfied:

\begin{itemize}
\item[(E.1$'$)] $\Ep[\exp(|\tilde{X}_{i j}|/B_n)] \leq 2$ {\em for all} $i=1,\dots,n$ {\em and} $j=1,\dots,m$,
\item[(E.2$'$)] $\Ep[ (\max_{1\leq j \leq m} |\tilde{X}_{ij}| / B_{n})^q] \leq 2$ {\em for all} $i=1,\dots,n$.
\end{itemize}

Conditions (M.1$'$), (M.2$'$), (E.1$'$), and (E.2$'$) are similar to those used in the previous section but they apply to  $\tilde{X}_1,\dots,\tilde{X}_n$ rather than to $X_1,\dots,X_n$.

Recall the definition of $\rho_n(\mA)$ in (\ref{eq: rho definition general}) and the definitions of $D_{n}^{(1)}$ and $D_{n,q}^{(2)}$ in (\ref{eq: definition D}).
An extension of Proposition \ref{cor: examples} leads to the following result.

\begin{proposition}[High Dimensional CLT for Simple Convex Sets]\label{cor: sparsely convex bodies}
Let $\mA$ be a subclass of $\mA^{\si}(a,d)$ such that conditions {\em (M.1$'$)}, {\em (M.2$'$)}, and {\em (E.1$'$)} are satisfied for every $A \in \mA$. Then
\begin{equation}\label{eq: best rate}
\rho_n(\mA) 
% \leq C\left(\frac{B_n^2\log^7(p n)}{n}\right)^{1/6},
\leq C D_{n}^{(1)},
\end{equation}
where the constant $C$ depends only on $a$, $b$, and $d$. 
If, instead of condition {\em (E.1$'$)}, condition {\em (E.2$'$)} is satisfied for every $A \in \mA$, then 
\begin{equation}\label{eq: best rate 2}
\rho_n(\mA) 
% \leq C\left[\left(\frac{B_n^2(\log^7 (p n)}{n}\right)^{1/6}+\left(\frac{B_n^2\log^3 p}{n^{1-2/q}}\right)^{1/3}\right],
\leq C \{ D_{n}^{(1)} + D_{n,q}^{(2)} \},
\end{equation}
where the constant $C$ depends only on $a$, $b$, $d$, and $q$.
\end{proposition}

It is worthwhile to mention that a notable example where the transformed variables $\tilde{X}_{i} =  (v'X_{i})_{v \in \mathcal{V}(A^{m})}$ satisfy condition (E.1$'$) is the case where each $X_{i}$ obeys a log-concave distribution. Recall that a Borel probability measure $\mu$ on $\R^{p}$ is {\em log-concave} if for any compact sets $A_{1},A_{2}$ in $\R^{p}$ and $\lambda \in (0,1)$, 
\[
\mu (\lambda A_{1}+(1-\lambda)A_{2}) \geq \mu (A_{1})^{\lambda}\mu (A_{2})^{1-\lambda}, 
\]
where $\lambda A_{1}+(1-\lambda)A_{2} = \{ \lambda x + (1-\lambda)y : x \in A_{1}, y \in A_{2} \}$. 
% We can indeed take $A_{1}$ and $A_{2}$ to be just Borel sets in $\R^{p}$ in the definition of log-concavity by regularity of Borel probability measures on $\R^{p}$; in that case the set $\lambda A_{1}+(1-\lambda) A_{2}$ need not to be Boel measurable, but since it is the image of the Borel set $A_{1} \times A_{2}$ in $\R^{2p}$ by the continuous map $\R^{2p} \ni (x,y) \mapsto \lambda x + (1-\lambda) y$ and thus an analytic set, it is measurable with respect to the completion of $\mu$ \citep[][Theorem 13.2.6]{Dudley02} and so $\mu (\lambda A_{1}+(1-\lambda)A_{2})$ makes sense even when $A_{1}$ and $A_{2}$ are just Borel measurable. 

\begin{corollary}[High Dimensional CLT for Simple Convex Sets with Log-concave Distributions]\label{cor: log-concave distributions}
Suppose that each $X_{i}$ obeys a centered log-concave distribution on $\R^p$ and that all the eigenvalues of $\Ep[X_i X_i']$ are bounded from below by a constant $k_1>0$ and from above by a constant $k_2\geq k_1$ for every $i=1,\dots,n$. Then
\[
\rho_n(\mA^{\si}(a,d)) \leq C n^{-1/6} \log^{7/6}(p n),
\]
where the constant $C$ depends only on $a,b,d,k_1$, and $k_2$.
\end{corollary}

\subsection{Sparsely convex sets}

We next consider classes of sparsely convex sets defined as follows.

\begin{definition}[Sparsely convex sets]\label{def: sparsely convex sets}
For integer $s>0$, we say that $A\subset\R^p$ is an {\em $s$-sparsely convex set} if there exist an integer $Q>0$ and convex sets $A_q \subset \R^p, q=1,\dots,Q$, such that $A=\cap_{q=1}^Q A_q$ and the indicator function of each $A_q$, $w\mapsto I(w\in A_q)$, depends at most on $s$ elements of its argument $w=(w_1,\dots,w_p)$ (which we call the  main components of $A_q$). We also say that $A=\cap_{q=1}^Q A_q$ is a sparse representation of $A$.
\end{definition}

Observe that for any $s$-sparsely convex set $A\subset \R^p$, the integer $Q$ in Definition \ref{def: sparsely convex sets} can be chosen to satisfy $Q\leq C_s^p\leq p^s$, where $C_s^p$ is the number of combinations of size $s$ from $p$ objects. Indeed, if we have a sparse representation $A=\cap_{q=1}^Q A_q$ for $Q>C_s^p$, then there are at least two sets $A_{q_{1}}$ and $A_{q_{2}}$ with the same main components, and hence we can replace these two sets by one convex set $A_{q_{1}}\cap A_{q_{2}}$ with the same main components; this procedure can be repeated until we have $Q \leq C_s^p$.

\begin{example}
The simplest example satisfying Definition \ref{def: sparsely convex sets} is a hyperrectangle as in (\ref{eq: rectangle}), which is a $1$-sparsely convex set. Another  example is the set
\[
A=\{ w \in \R^p: v_k' w \leq a_k \ \text{for all} \ k=1,\dots,m\}
\]
for some unit vectors $v_k \in \mathbb{S}^{p-1}$ and coefficients $a_k$, $k=1,\dots,m$. If the number of non-zero elements of each $v_k$ does not exceed $s$, this $A$ is an $s$-sparsely convex set.
Yet another example is the set
\[
A=\{ w \in \R^p:a_j\leq w_j\leq b_j \ \text{for all} \ j=1,\dots,p \ \text{and} \ w_1^2+w_2^2\leq c\}
\]
for some coefficients $-\infty\leq a_j\leq b_j\leq \infty$, $j=1,\dots,p$, and $0< c\leq \infty$. This $A$ is a $2$-sparsely convex set.
A more complicated example is the set
\[
A=\{ w \in \R^p:a_j\leq w_j\leq b_j,w_k^2+w_l^2\leq c_{k l}, \ \text{for all} \ j,k,l=1,\dots,p\}
\]
for some coefficients $-\infty\leq a_j\leq b_j\leq \infty$, $0< c_{k l}\leq \infty$, $j,k,l=1,\dots,p$. This $A$ is a $2$-sparsely convex set.  Finally, consider the set
\[
A = \{ w \in \R^p: \|(w_j)_{j \in J_k}\|^2  \leq t_k \ \text{for all} \ k=1,\dots,\kappa\},
\]
where $\{J_k\}$ are subsets of $\{1,\dots,p\}$ of fixed cardinality $s$, $\{t_k\}$  are thresholds in $(0, \infty)$, and $1 \leq \kappa \leq C_s^p$ is an integer.  This $A$ is an $s$-sparsely convex set.\qed
\end{example}

As the proof of Proposition \ref{cor: sparsely convex sets} below reveals, $s$-sparsely convex sets are closely related to simple convex sets. In particular, we can split any $s$-sparsely convex set $A\subset\mathbb R^p$ into $A\cap B$ and $A\cap B'$ for a cube $B = \{w\in\mathbb R^p: \max_{1\leq j\leq p}|w_j|\leq R\}$. Setting $R = p n^{5/2}$, it is easy to show that both $\Pr(S_n^X\in A\cap B')$ and $\Pr(S_n^Y\in A\cap B')$ are negligible. On the other hand, $A\cap B$ is a simple convex set with parameters $a=1$ and $d$ depending only $s$ as long as $A\cap B$ contains a ball of radius $1/n$, and if $A\cap B$ does not contain such a ball, both $\Pr(S_n^X\in A\cap B)$ and $\Pr(S_n^Y\in A\cap B)$ are also negligible.

Fix an integer $s>0$, and let $\mA^{\spa}(s)$ denote the class of all $s$-sparsely convex Borel sets in $\R^p$. We assume that the following condition is satisfied:

\begin{itemize}
\item[(M.1$''$)] $n^{-1} \sum_{i=1}^{n}\Ep[(v'X_{i})^2] \geq b$ {\em for all} $v\in\mathbb{S}^{p-1}$ {\em with} $\|v\|_0\leq s$.
\end{itemize}

Then we have the following proposition:
\begin{proposition}[High Dimensional CLT for Sparsely Convex Sets]\label{cor: sparsely convex sets}
Suppose that conditions {\em (M.1$''$)} and {\em (M.2)} are satisfied. Then under {\em (E.1)}, we have
\begin{equation}\label{eq: sparsely convex sets with mgf}
\rho_n(\mA^{\spa}(s)) 
% \leq C\left(\frac{B_n^2\log^7(p n)}{n}\right)^{1/6},
\leq C D_{n}^{(1)},
\end{equation}
where the constant $C$ depends only on $b$ and $s$; while under {\em (E.2)}, we have
\begin{equation}\label{eq: sparsely convex sets with finite moments}
\rho_n(\mA^{\spa}(s)) 
% \leq C\left[\left(\frac{B_n^2(\log^7 (p n)}{n}\right)^{1/6}+\left(\frac{B_n^2\log^3 p}{n^{1-2/q}}\right)^{1/3}\right],
\leq C \{ D_{n}^{(1)} + D_{n,q}^{(2)} \},
\end{equation}
where the constant $C$ depends only on $b$, $q$, and $s$.
\end{proposition}
\begin{remark}[Dependence on $s$]
In many applications, it may be of interest to consider $s$-sparsely convex sets with $s=s_n$ depending on $n$ and potentially growing to infinity: $s=s_n\to\infty$. It is therefore interesting to derive the optimal dependence of the constant $C$ in (\ref{eq: sparsely convex sets with mgf}) and (\ref{eq: sparsely convex sets with finite moments}) on $s$. We leave this question for future work. \qed
\end{remark}

%The following lemma shows than many sparsely convex sets are simple convex sets:
%We have the following lemma yielding a set of primitive conditions that suffice for conditions (C.j):
%\begin{remark}[On conditions of Lemma \ref{lem: verifying condition C.j}]
%The conditions of the lemma are sufficient to establish the bound (\ref{eq: m bound}) but they are not necessary. These conditions arise from the fact that our derivation relies upon the approximation result of Barvinok \cite{B13}. Similar bounds can often be established under different/weaker conditions using ad hoc arguments. For example, these assumptions can be clearly avoided in the case of the set in (\ref{eq: linear combination}).
%\end{remark}

%We conclude this section with a lemma relating conditions (M.2$'$), (E.1$'$), and (E.2$'$) to conditions (M.2), (E.1), and (E.2).

\section{Empirical and Multiplier Bootstrap Theorems}\label{sec: bootstrap CLT}
So far we  have shown that the probabilities $\Pr(S_n^X\in A)$ can be well approximated by the probabilities  $\Pr(S_n^Y\in A)$ under weak conditions for hyperrectangles $A \in \mA^{\re}$, simple convex sets $A\in\mA^{\si}(a,d)$, or sparsely convex sets $A\in\mA^{\spa}(s)$. 
In practice, however, the covariance matrix of $S_n^Y$ is typically unknown, and direct computation of  $\Pr(S_n^Y\in A)$ is infeasible. 
Hence, in this section, we derive high dimensional bootstrap theorems which allow us to approximate the probabilities $\Pr(S_n^Y\in A)$, and hence $\Pr(S_n^X\in A)$, by data-dependent techniques. We consider here multiplier and empirical bootstrap methods (we refer to \cite{PW93} for various versions of bootstraps). %We state main Theorems \ref{thm: multiplier bootstrap CLT} and \ref{thm: empirical bootstrap CLT} for the case of rectangles (multiplier and empirical bootstrap, respectively), and we extend these results for the case of simple convex sets in examples . Since the class of rectangles $\mA$ is a special case of a class of sparsely convex sets $\mA^{\si}$, we only consider the latter case in this section.

\subsection{Multiplier bootstrap}

We first consider the multiplier bootstrap. Let $e_1,\dots,e_n$ be a sequence of i.i.d. $N(0,1)$ random variables that are independent of $X_{1}^{n}= \{ X_{1},\dots,X_{n} \}$. Let $\bar{X}:=(\bar{X}_{1},\dots,\bar{X}_{p})':=\En[X_i]$, and consider the normalized sum:
\[
S_n^{eX}:=(S_{n1}^{eX},\dots,S_{np}^{eX})':=\frac{1}{\sqrt{n}}\sum_{i=1}^ne_i(X_i-\bar{X}).
\]
We are interested in bounding
\[
\rho_n^{MB}(\mA):=\sup_{A \in \mA}|\Pr(S_n^{eX}\in A \mid X_{1}^{n})-\Pr(S_n^Y\in A)|
\]
for $\mA=\mA^{\re}$, $\mA^{\spa}(s)$, or $\mA\subset \mA^{\si}(a,d)$. 

We begin with the case $\mA\subset \mA^{\si}(a,d)$. 
Let
\begin{align*}
\hat{\Sigma}:=n^{-1}\sum_{i=1}^n(X_i-\bar{X})(X_i-\bar{X})', \quad \Sigma:=n^{-1}\sum_{i=1}^n\Ep[X_iX_i'].
\end{align*}
Observe that $\Ep[S_{n}^{eX}(S_{n}^{eX})' \mid X_{1}^{n}] = \hat{\Sigma}$ and $\Ep[S_{n}^{Y}(S_{n}^{Y})'] = \Sigma$. For $\mA\subset \mA^{\si}(a,d)$, define
\[
\Delta_n(\mA):=\sup_{A\in\mA}\max_{v_1,v_2\in\mathcal{V}(A^m(A))}|v_1'(\hat{\Sigma}-\Sigma)v_2|.
\]
Then we have the following theorem for classes of simple convex sets.

\begin{theorem}[Abstract Multiplier Bootstrap Theorem for Simple Convex Sets]\label{thm: multiplier bootstrap CLT}
Let $\mA$ be a subclass of $\mA^{\si}(a,d)$ such that condition {\em (M.1$'$)} is satisfied for every $A \in \mA$. Then for every constant $\overline{\Delta}_n > 0$, on the event $\Delta_n(\mA)\leq \overline{\Delta}_n$, we have 
\[
\rho_n^{MB}(\mA)\leq C\left \{ \overline{\Delta}_n^{1/3}\log^{2/3} (pn) + n^{-1}\log^{1/2} (pn)\right \},
\]
where the constant $C$ depends only on $a,b$, and $d$.
\end{theorem}

\begin{remark}[Case of hyperrectangles]
From the proof of Theorem \ref{thm: multiplier bootstrap CLT}, we have the following bound when $\mA = \mA^{\re}$: under (M.1), for every constant $\overline{\Delta}_n > 0$, on the event $\Delta_{n,r}\leq \overline{\Delta}_n$, we have
\[
\rho_n^{MB}(\mA^{\re})\leq C\overline{\Delta}_n^{1/3}\log^{2/3} p,
\]
where the constant $C$ depends only on $b$, and $\Delta_{n,r}$ is defined by 
\[
\Delta_{n,r}=\max_{1\leq j,k\leq p}|\hat{\Sigma}_{j k}-\Sigma_{jk}|,
\]
where $\hat{\Sigma}_{jk}$ and $\Sigma_{jk}$ are the $(j,k)$-th elements of $\hat{\Sigma}$ and $\Sigma$, respectively.
\qed
\end{remark}

Next, we derive more explicit bounds on $\rho_{n}^{MB} (\mA)$ for $\mA\subset \mA^{\si}(a,d)$ under suitable moment conditions as in the previous section. %For brevity of the paper, we state the corollary for the case of simple convex sets 
We will consider sets $A\in\mA^{\si}(a,d)$ that satisfy the following condition:
%We will need to strengthen condition (C) and will assume that all sets $A$ in $\mA^{\si}$ satisfy the following condition:
\medskip
\begin{itemize}
\item[(S)] \textit{The set $A^m=A^m(A)$ satisfies $\|v\|_0\leq s$ for all $v\in\mathcal{V}(A^m)$.}
\end{itemize}
Condition (S) requires that the outward unit normal vectors to the hyperplanes forming the $m$-generated convex set $A^m=A^m(A)$ are sparse.
Assuming that (S) is satisfied for all $A\in\mA\subset \mA^{\si}(a,d)$ helps to control $\Delta_{n}(\mA)$. 

For $\alpha \in (0,e^{-1})$, define 
\[
D_{n}^{(1)}(\alpha) = \left(\frac{B_n^2(\log^5 (p n))\log^2(1/\alpha)}{n}\right)^{1/6}, \ D_{n,q}^{(2)}(\alpha) = \left (\frac{B_n^2\log^3 (pn)}{\alpha^{2/q}n^{1-2/q}} \right)^{1/3}.
\]
Then we have the following proposition. 
%An application of this theorem under conditions (E.1) or (E.2) leads to the following corollary. 
\begin{proposition}[Multiplier Bootstrap for Simple Convex Sets]\label{cor: MB corollary}
Let $\alpha\in(0,e^{-1})$ be a constant, and let $\mA$ be a subclass of $\mA^{\si}(a,d)$ such that conditions {\em (S)} and {\em (M.1$'$)} are satisfied for every $A \in \mA$. In addition, suppose that condition {\em (M.2)} is satisfied. Then under {\em (E.1)}, we have with probability at least $1-\alpha$,
\[
\rho_n^{MB}(\mA)\leq CD_{n}^{(1)}(\alpha),
\]
where the constant $C$ depends only on $a,b,d$ and $s$; while under {\em (E.2)}, we have with probability at least $1-\alpha$,
\[
\rho_n^{MB}(\mA)\leq C\{ D_{n}^{(1)}(\alpha) + D_{n,q}^{(1)}(\alpha) \},
\]
where the constant $C$ depends only on $a,b,d,q$, and $s$.
\end{proposition}

\begin{remark}[Bootstrap theorems in a.s. sense]
Proposition \ref{cor: MB corollary} leads to the following multiplier bootstrap theorem in the a.s. sense. 
% Let $Z,Z_1,Z_2,\dots$ be i.i.d. random variables taking values in a measurable space $(S,\mathcal{S})$. For each $n\geq 4$, let $\mathcal{F}_n$ be a class consisting of $p=p_n=|\mathcal{F}_n|\geq 2$ measurable functions $S\to\R$. Consider random vectors in $\R^p$, $X_i:=X_{i,n}:=(f(Z_i)-\Ep[f(Z)])_{f\in\mathcal{F}_n}$, for $i=1,\dots,n$. Also define $S_n^X$, $S_n^{e X}$, $S_n^{X^*}$, $\rho_n^{M B}(\mA^{\si})$, and $\rho_n^{E B}(\mA^{\si})$ as above. 
Suppose that $\mA$ is a subclass of $\mA^{\si}(a,d)$ as in Proposition \ref{cor: MB corollary} and that (M.2) is satisfied. We allow $p=p_{n} \to \infty$ and $B_{n} \to \infty$ as $n \to \infty$  but assume that $a,b,d,q,s$ are all fixed. 
Then by applying Proposition \ref{cor: MB corollary} with $\alpha=\alpha_n=n^{-1}(\log n)^{-2}$, together with the Borel-Cantelli lemma (note that $\sum_{n=4}^\infty n^{-1}(\log n)^{-2}<\infty$), we have with probability one
\[
\label{eq: almost sure bootstrap}
\rho_n^{M B}(\mA) 
= \begin{cases}
O\{ D_{n}^{(1)}(\alpha_n) \} & \text{under (E.1)} \\
O\{ D_{n}^{(1)}(\alpha_n) \vee D_{n,q}^{(2)}(\alpha_{n}) \} & \text{under (E.2)},
\end{cases}
\]
and it is routine to verify that $D_{n}^{(1)}(\alpha_n) = o(1)$ if $B_{n}^{2} \log^{7} (pn) = o(n)$, and $D_{n,q}^{(2)}(\alpha_{n})=o(1)$ if $B_n^2(\log^3 (p n))\log^{4/q}n=o(n^{1-4/q})$. Similar conclusions also follow from other propositions and corollaries below dealing with different classes of sets and approximations based on multiplier and empirical bootstraps.\qed
% if $B_n^2\log^7(p n)/n=o(1)$. Similarly, under conditions (C), (M.1$'$), (M.2$'$), and (E.2$'$), (\ref{eq: almost sure bootstrap}) holds if $B_n^2\log^7(p n)/n=o(1)$ and $$.
\end{remark}

When each $X_{i}$ obeys a log-concave distribution, we have the following corollary analogous to Corollary \ref{cor: log-concave distributions}. In this case, instead of condition (S), we will assume that $\mA\subset \mA^{\si}(a,d)$ is such that the cardinality of the set $\cup_{A\in\mA}\mathcal{V}(A^m(A))$ is at most $(p n)^d$.

\begin{corollary}[Multiplier Bootstrap for Simple Convex Sets with Log-concave Distributions]
\label{cor: MB corollary log-concave distributions}
Let $\alpha\in(0,e^{-1})$ be a constant, and let $\mA$ be a subclass of $\mA^{\si}(a,d)$ such that the cardinality of the set $\cup_{A\in\mA}\mathcal{V}(A^m(A))$ is at most $(p n)^d$. Suppose that each $X_{i}$ obeys a centered log-concave distribution on $\R^p$ and that all the eigenvalues of $\Ep[X_i X_i']$ are bounded from below by a constant $k_1>0$ and from above by a constant $k_2\geq k_1$ for every $i=1,\dots,n$. Then with probability at least $1-\alpha$,
\[
\rho_n^{MB}(\mA)\leq C n^{-1/6} (\log^{5/6}(p n))\log^{1/3} (1/\alpha), 
\]
where the constant $C$ depends only on $a,d,k_1$, and $k_2$.
\end{corollary}

When $\mA=\mA^{\re}$, we have the following corollary.
\begin{corollary}[Multiplier Bootstrap for Hyperrectangles]\label{cor: MB for rectangles explicit}
Let $\alpha\in(0,e^{-1})$ be a constant, and suppose that conditions {\em (M.1)} and {\em (M.2)} are satisfied. Then under {\em (E.1)}, we have with probability at least $1-\alpha$,
\[
\rho_n^{MB}(\mA^{\re})\leq CD_{n}^{(1)}(\alpha),
\]
where the constant $C$ depends only on $b$; while under {\em (E.2)}, we have with probability at least $1-\alpha$,
\[
\rho_n^{MB}(\mA^{\re})\leq C\{ D_{n}^{(1)}(\alpha) + D_{n,q}^{(1)}(\alpha) \},
\]
where the constant $C$ depends only on $b$ and $q$.
\end{corollary}

Finally, we derive explicit bounds on $\rho_{n}^{MB}(\mA)$ in the case where $\mA$ is the class of all $s$-sparsely convex sets: $\mA = \mA^{\spa}(s)$.
\begin{proposition}[Multiplier Bootstrap for Sparsely Convex Sets]
\label{cor: MB corollary sparsely convex sets}
Let $\alpha\in(0,e^{-1})$ be a constant. Suppose that conditions {\em (M.1$''$)} and {\em (M.2)} are satisfied. Then under {\em (E.1)}, we have with probability at least $1-\alpha$,
\begin{equation}\label{eq: MB result E1 sparse}
\rho_n^{MB}(\mA^{\spa}(s)) \leq CD_{n}^{(1)}(\alpha),
\end{equation}
where the constant $C$ depends only on $b$ and $s$; while under {\em (E.2)}, we have with probability at least $1-\alpha$,
\begin{equation}\label{eq: MB result E2 sparse}
\rho_n^{MB}(\mA^{\spa}(s))\leq C\{ D_{n}^{(1)}(\alpha) + D_{n,q}^{(2)}(\alpha) \},
\end{equation}
where the constant $C$ depends only on $b,s$, and $q$.
\end{proposition}

\subsection{Empirical bootstrap}

Here we consider the empirical bootstrap. For brevity, we only consider the case $\mA=\mA^{\re}$. Let $X_1^*,\dots,X_n^*$ be i.i.d. draws from the empirical distribution of $X_1,\dots,X_n$. Conditional on $X_{1}^{n} = \{ X_{1},\dots,X_{n} \}$, $X_1^*,\dots,X_n^*$ are i.i.d. with mean $\bar{X}=\En[X_i]$. Consider the normalized sum:
\[
S_n^{X^*}:=(S_{n1}^{X^*},\dots,S_{np}^{X^*})':=\frac{1}{\sqrt{n}}\sum_{i=1}^n (X_i^*-\bar{X}).
\]
We are interested in bounding
\[
\rho_n^{EB}(\mA):=\sup_{A\in\mA}|\Pr(S_n^{X^*} \in A \mid X_{1}^{n})-\Pr(S_n^Y\in A)|
\]
for $\mA=\mA^{\re}$.
To state the bound, define
\[
\hat{L}_n:=\max_{1\leq j\leq p}\sum_{i=1}^n |X_{ij}-\bar{X}_{j}|^3/n,
\]
which is an empirical analog of $L_n$, 
and for $\phi\geq 1$, define
\begin{align*}
&\hat{M}_{n,X}(\phi):=n^{-1}\sum_{i=1}^n\max_{1\leq j\leq p}|X_{i j}-\bar{X}_{j}|^31\left\{\max_{1\leq j\leq p}|X_{i j}-\bar{X}_{j}|>\sqrt{n}/(4\phi\log p)\right\},\\
&\hat{M}_{n,Y}(\phi):=\Ep\left[\max_{1\leq j\leq p}|S_{n j}^{e X}|^31\left\{\max_{1\leq j\leq p}|S_{n j}^{e X}|>\sqrt{n}/(4\phi\log p)\right\} \mid X_{1}^{n} \right],
\end{align*}
which are empirical analogs of $M_{n,X}(\phi)$ and $M_{n,Y}(\phi)$, respectively. Let 
\[
\hat{M}_n(\phi):=\hat{M}_{n,X}(\phi)+\hat{M}_{n,Y}(\phi).
\]
We have the following theorem.
\begin{theorem}[Abstract Empirical Bootstrap Theorem]\label{thm: empirical bootstrap CLT}
For arbitrary positive constants $b$, $\overline{L}_{n}$, and $\overline{M}_n$, the inequality 
\[
\rho_n^{EB}(\mA^{\re})\leq \rho_{n}^{MB}(\mA^{\re}) + K_1 \left[ \left(\frac{\overline{L}_{n}^2\log^7 p}{n}\right)^{1/6}+\frac{\overline{M}_n}{\overline{L}_{n}} \right]
\]
holds on the event 
\[
\{ \En[(X_{ij}-\bar{X}_{j})^2]\geq b \ \text{for all} \ j=1,\dots,p \} \cap \{  \hat{L}_{n} \leq \overline{L}_{n} \} \cap \{ \hat{M}_{n}(\phi_n) \leq \overline{M}_n \},
\]
where
\[
\phi_n:=K_2\left(\frac{\overline{L}_{n}^2\log^4 p}{n}\right)^{-1/6}.
\]
Here $K_1, K_2 > 0$ are constants that depend only on $b$.
\end{theorem}

As for the multiplier bootstrap case, we next derive explicit bounds on $\rho_{n}^{EB}(\mA^{\re})$ under suitable moment conditions. %Here we only state the results for classes of simple convex sets $\mA\subset\mA^{\si}(a,d)$ but note that the same result applies to the case of rectangles $\mA=\mA^{\re}$.

\begin{proposition}[Empirical Bootstrap for Hyperrectangles]\label{cor: EB corollary}
Let $\alpha\in(0,e^{-1})$ be a constant, and suppose that conditions {\em (M.1)} and {\em (M.2)} are satisfied. In addition, suppose that $\log(1/\alpha)\leq K\log(p n)$ for some constant $K$. Then under {\em (E.1)}, we have with probability at least $1-\alpha$,
\begin{equation}\label{eq: EB example E1}
\rho_n^{EB}(\mA^{\re}) \leq CD_{n}^{(1)},
\end{equation}
where the constant $C$ depends only on $b$ and $K$;
while under {\em (E.2)}, we have with probability at least $1-\alpha$,
\begin{equation}\label{eq: EB example E2}
\rho_n^{EB}(\mA^{\re})\leq C \{ D_{n}^{(1)} + D_{n,q}^{(2)}(\alpha) \},
\end{equation}
where the constant $C$ depends only on $b,q$, and $K$.
\end{proposition}

%When each $X_{i}$ obeys a log-concave distribution, we have the following corollary.

%\begin{corollary}[Empirical Bootstrap for Simple Convex Sets with Log-concave Distributions]
%\label{cor: EB corollary log-concave distributions}
%Let $\alpha\in(0,e^{-1})$ be a constant and suppose that $\log(1/\alpha)\leq K\log(p n)$ for some other constant $K$. Moreover, suppose that all the assumptions in Corollary \ref{cor: MB corollary log-concave distributions} are satisfied. Then with probability at least $1-\alpha$,
%\[
%\rho_n^{EB}(\mA) \leq C  n^{-1/6}\log^{7/6}(pn),
%\]
%where the constant $C$ depends only on $a,d,k_1,k_2$, and $K$.
%\end{corollary}

%We conclude this section with an application of Corollaries \ref{cor: MB corollary} and \ref{cor: EB corollary} to log-concave distributions:
%\begin{corollary}[Bootstrap CLTs for log-concave distributions]\label{cor: bootstrap log-concave distribution}
%\end{corollary}

\section{Key Lemma}
\label{sec: induction lemma}

In this section, we state a lemma that plays a key role in the proof of our high dimensional CLT for hyperrectangles (Theorem \ref{thm: main}). Define
\[
\varrho_{n}:=\sup_{y \in \R^p, v\in[0,1]} |\Pr\left(\sqrt{v}S^X_{n}+\sqrt{1-v}S^Y_{n} \leq y \right)-\Pr(S^Y_{n} \leq y)|,
\]
where the random vectors $Y_1,\dots,Y_n$ are assumed to be independent of the random vectors $X_1,\dots,X_n$,
and recall that $M_n(\phi):=M_{n,X}(\phi)+M_{n,Y}(\phi)$ for $\phi\geq 1$.
%Clearly, $\rho_n\leq \tilde{\rho}_n$, and so it suffices to bound $\tilde{\rho}_n$. 
The lemma below provides a bound on $\varrho_{n}$.

%Here is the main result of this section.
\begin{lemma}[Key Lemma]\label{lem: induction lemma}
Suppose that there exists some constant $b>0$ such that $n^{-1}\sum_{i=1}^n\Ep[X_{i j}^2]\geq b$ for all $j=1,\dots,p$. Then $\varrho_{n}$ satisfies the following inequality for all $\phi\geq 1$:
\[
\varrho_{n}
\lesssim
\frac{\phi^2\log^2 p}{n^{1/2}}\left \{ \phi L_n \varrho_{n}+L_n\log^{1/2} p+\phi M_n(\phi)\right \} +\frac{\log^{1/2} p}{\phi}
\]
up to a constant $K$ that depends only on $b$.
\end{lemma}
Lemma \ref{lem: induction lemma} has an immediate corollary. Indeed, define
\[
\varrho'_{n}:=\sup_{A \in \mA^{\re}, v\in[0,1]} |\Pr(\sqrt{v}S^X_{n}+\sqrt{1-v}S^Y_{n}\in A)-\Pr(S^Y_{n}\in A)|
\]
where $\mA^{\re}$ is the class of all hyperrectangles in $\R^p$.
Then we have:
\begin{corollary}\label{cor: induction lemma}
Suppose that there exists some constant $b>0$ such that $n^{-1}\sum_{i=1}^n\Ep[X_{i j}^2]\geq b$ for all $j=1,\dots,p$. Then $\varrho'_{n}$ satisfies the following inequality for all $\phi\geq 1$:
\[
\varrho'_{n}
\lesssim
\frac{\phi^2\log^2 p}{n^{1/2}}\left\{ \phi L_n \varrho'_{n}+L_n \log^{1/2} p + \phi M_n(2\phi)\right \}+\frac{\log^{1/2} p}{\phi}
\]
up to a constant $K'$ that depends only on $b$.
\end{corollary}

\appendix

\section{Anti-concentration inequalities}

One of the main ingredients of the proof of Lemma \ref{lem: induction lemma} (and the proofs of the other results indeed) is the following anti-concentration inequality due to Nazarov \cite{N03}. 

\begin{lemma}[Nazarov's inequality, \cite{N03}]
\label{lem: anti-concentration}
Let $Y=(Y_1,\dots,Y_p)'$ be a centered Gaussian random vector in $\R^p$ such that $\Ep[Y_j^2]\geq b$ for all $j=1,\dots,p$ and some constant $b>0$. 
Then for every $y \in \R^{p}$ and $a>0$,
\[
\Pr(Y \leq y+a)-\Pr(Y \leq y)\leq Ca\sqrt{\log p},
\]
where $C$ is a constant depending only on $b$.
\end{lemma}

\begin{remark}
This inequality is less sharp than the dimension-free anti-concentration bound $C a \Ep [\max_{1 \leq j \leq p} Y_{j}]$ proved in \cite{CCK12c} for the case of max hyperrectangles.  However, the former inequality allows  for more general hyperrectangles  than the latter.  The difference in sharpness for the case of max-hyperrectangles arises due to dimension-dependence $\sqrt{\log p}$,  in particular the term $\sqrt{\log p}$ can be much larger than $\Ep [\max_{1 \leq j \leq p}Y_{j}]$. This also makes the anti-concentration bound in \cite{CCK12c}  more relevant for the study of suprema of Gaussian processes indexed by infinite classes.  It is an interesting question whether one could establish a dimension-free anti-concentration bound similar to that  in \cite{CCK12c} for classes of hyperrectangles other than max hyperrectangles. \qed
\end{remark}

\begin{proof}[Proof of Lemma \ref{lem: anti-concentration}]
Let $\Sigma=\Ep[YY']$; then $Y$ has the same distribution as $\Sigma^{1/2}Z$ where $Z$ is a standard Gaussian random vector.
Write $\Sigma^{1/2}=(\sigma_1,\dots,\sigma_p)'$ where each $\sigma_{j}$ is a $p$-dimensional vector. Note that $\| \sigma_{j} \| = (\Ep[Y_{j}^{2}])^{1/2} \geq b^{1/2}$. Then 
\begin{align*}
&\Pr(Y\leq y+a)=\Pr(\Sigma^{1/2}Z \leq y+a) \\
&=\Pr((\sigma_j/\| \sigma_j \|)'Z\leq (y_j+a)/\| \sigma_j \| \ \text{for all} \ j=1,\dots,p),
\end{align*}
and similarly
\[
\Pr(Y \leq y)=\Pr((\sigma_j/\| \sigma_j \|)'Z\leq y_j/\| \sigma_j \| \ \text{for all} \ j=1,\dots,p).
\]
Since $Z$ is a standard Gaussian random vector, and $a/\|\sigma_j\|\leq a/b^{1/2}$ for all $j=1,\dots,p$, the assertion follows from Theorem 20 in \cite{KDS08}, whose proof the authors credit to Nazarov \cite{N03}.
\end{proof}

We will use another anti-concentration inequality by \cite{N03} in the proofs for Sections \ref{sec: friendly sets} and \ref{sec: bootstrap CLT}, which is an extension of Theorem 4 in \cite{B93}.

\begin{lemma}
\label{lem: Nazarov anti-concentration2}
Let $A$ be a $p \times p$ symmetric positive definite matrix, and let $\gamma_{A} = N(0,A^{-1})$. Then there exists a universal constant $C > 0$ such that for every convex set $Q\subset \R^{p}$, and every $h_1,h_2>0$,
\[
\frac{\gamma_{A}(Q^{h_1} \setminus Q^{-h_2})}{h_1+h_2} \leq C \sqrt{\| A \|_{HS}},
\]
where $\| A \|_{HS}$ is the Hilbert-Schmidt norm of $A$, $Q^{h} = \{x\in \mathbb{R}^p: \rho(x,Q)\leq h\}$, $Q^{-h}=\{x\in\mathbb{R}^p: B(x,h)\subset Q\}$, $B(x,h)=\{y\in\mathbb{R}^p: \|y-x\|\leq h\}$, and $\rho(x,Q)=\inf_{y\in Q}\|y-x\|$.
\end{lemma}

\begin{proof}
It is proven in \cite{N03} that for every convex set $Q\subset \mathbb{R}^p$ and every $h>0$,
\[
\frac{\gamma_{A}(Q^{h} \setminus Q)}{h} \leq C \sqrt{\| A \|_{HS}}.
\]
Therefore, the asserted claim follows from the arguments in Proposition 2.5 of \cite{CF11} or in Section 1.3 of \cite{BR86}.
\end{proof}

\section{Proof for Section \ref{sec: induction lemma}}
\label{sec: induction lemma proof}

We begin with stating the following variants of Chebyshev's association inequality.
\begin{lemma}
\label{lem: correlations}
Let $\varphi_{i}: \R \to [0,\infty), \ i=1,2$ be non-decreasing functions, and let $\xi_{i}, i=1,2$ be independent real-valued random variables. Then
\begin{align}
&\Ep[\varphi_{1}(\xi_{1})]\Ep[\varphi_{2}(\xi_{1})] \leq \Ep[\varphi_{1}(\xi_{1})\varphi_{2}(\xi_{1})],\label{eq: simp1}\\
&\Ep[\varphi_{1}(\xi_{1})]\Ep[\varphi_{2}(\xi_2)] \leq \Ep[\varphi_{1} (\xi_1)\varphi_{2}(\xi_1)]+\Ep[\varphi_{1}(\xi_2)\varphi_{2}(\xi_2)],\label{eq: simp2}\\
&\Ep[\varphi_1(\xi_1)\varphi_{2}(\xi_2)] \leq \Ep[\varphi_{1}(\xi_1)\varphi_{2}(\xi_1)]+\Ep[\varphi_{1}(\xi_2)\varphi_{2}(\xi_2)], \label{eq: simp3}
\end{align}
where we assume that all the expectations exist and are finite. Moreover, (\ref{eq: simp3}) holds without independence of $\xi_1$ and $\xi_2$. 
\end{lemma}
\begin{proof}[Proof of Lemma \ref{lem: correlations}]
The inequality (\ref{eq: simp1}) is Chebyshev's association inequality; see Theorem 2.14 in \cite{BLM2013}. Moreover, since $\xi_1$ and $\xi_2$ are independent, (\ref{eq: simp2}) follows from (\ref{eq: simp3}). In turn, (\ref{eq: simp3}) follows from
% \begin{align*}
\begin{align*}
&\Ep[\varphi_{1}(\xi_1)\varphi_{2}(\xi_2)] \leq \Ep[\varphi_{1}(\xi_1)\varphi_{2}(\xi_2)]+\Ep[\varphi_{2}(\xi_1)\varphi_{1}(\xi_2)] \\
&\quad \leq \Ep[\varphi_{1}(\xi_1)\varphi_{2}(\xi_1)]+\Ep[\varphi_{1}(\xi_2)\varphi_{2}(\xi_2)],
\end{align*}
where the first inequality follows from the fact that $\varphi_{2}(\xi_1)\varphi_{1}(\xi_2)\geq 0$, and the second inequality follows from rearranging the terms in the following inequality:
\[
\Ep[(\varphi_{1}(\xi_1)-\varphi_{1}(\xi_2))(\varphi_{2}(\xi_1)-\varphi_{2}(\xi_2))]\geq 0,
\]
which follows from monotonicity of $\varphi_{1}$ and $\varphi_{2}$. 
\end{proof}

\subsection*{Proof of Lemma \ref{lem: induction lemma}}
The proof relies on a Slepian-Stein method developed in \cite{CCK12a}. 
Here the notation $\lesssim$ means that the left-hand side is bounded by the right-hand side up to some constant depending only on $b$. 

We begin with preparing some notation. Let $W_1,\dots,W_n$ be a copy of $Y_1,\dots,Y_n$. Without loss of generality, we may assume that $X_1,\dots,X_n$, $Y_1,\dots,Y_n$, and $W_1,\dots,W_n$ are independent. Consider $S^W_{n}:=n^{-1/2}\sum_{i=1}^n W_i$.
Then $\Pr(S^Y_{n} \leq y)=\Pr(S^W_{n} \leq y)$, so that
\[
\varrho_{n}=\sup_{y\in\R^p,v\in[0,1]}|\Pr \left(\sqrt{v}S^X_{n}+\sqrt{1-v}S^Y_{n} \leq y \right)-\Pr(S^W_{n} \leq y)|.
\]
Pick any $y \in \R^p$ and $v \in [0,1]$.
Let $\beta:=\phi\log p$, 
and define the function
\[
F_\beta(w):=\beta^{-1} \log\left( {\textstyle \sum}_{j=1}^p\exp \left( \beta (w_{j}-y_{j}) \right) \right), \ w \in \R^{p}.
\]
The function $F_\beta(w)$ has the following property:
\begin{equation}
\label{eq: F properties}
0 \leq F_\beta(w)-\max_{1\leq j\leq p}(w_j-y_j)\leq \beta^{-1}\log p=\phi^{-1}, \ \text{for all} \ w \in \R^{p}.
\end{equation}
Pick a thrice continuously differentiable function $g_0:\R \to [0,1]$ whose derivatives up to the third order are all bounded such that $g_{0}(t)=1$ for $t \leq 0$ and $g_{0}(t)=0$ for $t\geq 1$. Define $g(t):=g_{0}(\phi t), t \in \R$, and 
\[
m(w):=g(F_\beta(w)), \ w \in \R^{p}.
\]
For brevity of notation, we will use indices to denote partial derivatives of $m$; for example, $\partial_j\partial_k\partial_l m=m_{j k l}$. The function $m(w)$ has the following properties established in Lemmas A.5 and A.6 of \cite{CCK12a}: for every $j,k,l=1,\dots,p$, there exists a function $U_{j k l}(w)$ such that
\begin{align}
&|m_{jkl}(w)|\leq U_{jkl}(w), \label{eq: property1}\\
&\textstyle{\sum_{j,k,l=1}^p} U_{jkl}(w) \lesssim (\phi^3+\phi\beta+\phi\beta^2)\lesssim \phi\beta^2, \label{eq: property2}\\
&U_{jkl}(w)\lesssim U_{jkl}(w+\tilde{w}) \lesssim U_{jkl}(w), \label{eq: property3}
\end{align}
where the inequalities (\ref{eq: property1}) and (\ref{eq: property2}) hold for all $w\in\R^p$, and the inequality (\ref{eq: property3}) holds for all $w,\tilde{w}\in\R^p$ with $\max_{1\leq j\leq p}|\tilde{w}_j|\beta\leq 1$ (formally, \cite{CCK12a} only considered the case where $y=(0,\dots,0)'$ but the extension to $y \in \R^p$ is trivial). 
Moreover, define the functions
\begin{align}
&h(w,t) := 1\left\{-\phi^{-1}-t/\beta< \max_{1\leq j\leq p}(w_j-y_j)\leq \phi^{-1}+t/\beta\right\}, \ w \in \R^{p}, t > 0, \label{eq: def iota} \\
&\omega(t) :=\frac{1}{\sqrt{t}\wedge\sqrt{1-t}}, \ t \in (0,1). \notag
\end{align}

The proof consists of two steps. In the first step, we show that
\begin{equation}\label{eq: to prove lemma 5.1}
|\Ep[\mathcal{I}_n]|
 \lesssim 
\frac{\phi^2\log^2 p}{n^{1/2}}\left(\phi L_n \varrho_{n}+L_n \log^{1/2} p+\phi M_n(\phi)\right)%+\frac{(\log p)^{1/2}}{\phi},
\end{equation}
where
\[
\mathcal{I}_n:=m(\sqrt{v}S^X_{n}+\sqrt{1-v}S^Y_{n})-m(S^W_{n}).
\]
In the second step, we combine this bound with Lemma \ref{lem: anti-concentration} to complete the proof.

\medskip

{\bf Step 1}. Define the Slepian interpolant
\[
Z(t):=\sum_{i=1}^n Z_i(t), \ t \in [0,1],
\]
where
\[
Z_i(t):=\frac{1}{\sqrt{n}}\left\{ \sqrt{t}(\sqrt{v}X_i+\sqrt{1-v}Y_i)+\sqrt{1-t}W_i \right \}.
\]
Note that $Z(1)=\sqrt{v}S^X_{n}+\sqrt{1-v}S^Y_{n}$ and $Z(0)=S^W_{n}$, 
and so
\begin{equation}\label{eq: interpolation}
\mathcal{I}_n=m(\sqrt{v}S^X_{n}+\sqrt{1-v}S^Y_{n})-m(S^W_{n})=\int_0^1 \frac{d m(Z(t))}{d t}d t.
\end{equation}
Denote by $Z^{(i)}(t)$ the Stein leave-one-out term for $Z(t)$:
\[
Z^{(i)}(t):=Z(t)-Z_i(t).
\]
Finally, define
\[
\dot{Z}_i(t):=\frac{1}{\sqrt{n}}\left \{ \frac{1}{\sqrt{t}}(\sqrt{v}X_i+\sqrt{1-v}Y_i)-\frac{1}{\sqrt{1-t}}W_i \right \}.
\]
For brevity of notation, we omit the argument $t$; that is, we write $Z=Z(t)$, $Z_i=Z_i(t)$, $Z^{(i)}=Z^{(i)}(t)$, and $\dot{Z}_i=\dot{Z}_i(t)$.

Now, from (\ref{eq: interpolation}) and Taylor's theorem, we have
\[
\Ep[\mathcal{I}_n]=\frac{1}{2}\sum_{j=1}^p\sum_{i=1}^n\int_0^1\Ep[m_j(Z)\dot{Z}_{i j}]d t=\frac{1}{2}(I+II+III),
\]
where
\begin{align*}
&I:=\sum_{j=1}^p\sum_{i=1}^n\int_0^1\Ep[m_j(Z^{(i)})\dot{Z}_{i j}]d t,\\
&II:=\sum_{j,k=1}^p\sum_{i=1}^n\int_0^1\Ep[m_{j k}(Z^{(i)})\dot{Z}_{i j}Z_{i k}]d t,\\
&III:=\sum_{j,k,l=1}^p\sum_{i=1}^n\int_0^1\int_0^1(1-\tau)\Ep[m_{j k l}(Z^{(i)}+\tau Z_i)\dot{Z}_{i j}Z_{i k}Z_{i l}]d\tau d t.
\end{align*}
By independence of $Z^{(i)}$ from $\dot{Z}_{ij}$ together with $\Ep[\dot{Z}_{i j}]=0$, we have $I=0$. Also, by independence of $Z^{(i)}$ from $\dot{Z}_{i j}Z_{i k}$ together with
\begin{align*}
\Ep[\dot{Z}_{i j}Z_{i k}]&=\frac{1}{n}\Ep\left[(\sqrt{v}X_{i j}+\sqrt{1-v}Y_{i j})(\sqrt{v}X_{i k}+\sqrt{1-v}Y_{i k})-W_{i j}W_{i k}\right]\\
&=\frac{1}{n}\Ep[v X_{i j}X_{i k}+(1-v)Y_{i j}Y_{i k}-W_{i j}W_{i k}]=0,
\end{align*}
we have $II=0$. Therefore, it suffices to bound $III$.

To this end, let
\[
\chi_i:=1\left\{\max_{1\leq j\leq p}|X_{i j}|\vee |Y_{i j}|\vee |W_{i j}|\leq \sqrt{n}/(4\beta)\right\}, \ i=1,\dots,n
\]
and decompose $III$ as $III=III_1+III_2$, where
\begin{align*}
&III_1:=\sum_{j,k,l=1}^p\sum_{i=1}^n\int_0^1\int_0^1(1-\tau)\Ep[\chi_i m_{jkl}(Z^{(i)}+\tau Z_i)\dot{Z}_{ij}Z_{i k}Z_{il}]d\tau d t,\\
&III_2:=\sum_{j,k,l=1}^p\sum_{i=1}^n\int_0^1\int_0^1(1-\tau)\Ep[(1-\chi_i)m_{jkl}(Z^{(i)}+\tau Z_i)\dot{Z}_{ij}Z_{i k}Z_{il}]d\tau d t.
\end{align*}
We shall bound $III_1$ and $III_2$ separately. For $III_2$, we have
\begin{align}
|III_2|&\leq \sum_{j,k,l=1}^p\sum_{i=1}^n\int_0^1\int_0^1\Ep[(1-\chi_i) U_{j k l}(Z^{(i)}+\tau Z_i)|\dot{Z}_{ij}Z_{ik}Z_{il}|]d\tau dt \notag\\
&\lesssim \phi\beta^2\sum_{i=1}^n\int_0^1 \Ep[(1-\chi_i)\max_{1\leq j,k,l\leq p}|\dot{Z}_{ij}Z_{ik}Z_{il}|] dt \notag\\
&\lesssim \frac{\phi\beta^2}{n^{3/2}}\sum_{i=1}^n\int_0^1 \omega(t)\Ep[(1-\chi_i)\max_{1\leq j\leq p}|X_{ij}|^3\vee|Y_{ij}|^3\vee|W_{i j}|^3] dt, \label{eq: integral}
\end{align}
where the first and the second inequalities follow from (\ref{eq: property1}) and (\ref{eq: property2}), respectively. 
Moreover, by letting $\mathcal{T}=\sqrt{n}/(4\beta)$ and using the union bound, we have
\[
1-\chi_i\leq 1\left\{\max_{1\leq j\leq p}|X_{ij}|>\mathcal{T}\right\}+1\left\{\max_{1\leq j\leq p}|Y_{ij}|>\mathcal{T}\right\}+1\left\{\max_{1\leq j\leq p}|W_{ij}|>\mathcal{T}\right\}.
\]
Hence, using the inequality
\[
\max_{1\leq j\leq p}|X_{ij}|^3\vee|Y_{ij}|^3\vee|W_{i j}|^3\leq \max_{1\leq j\leq p}|X_{ij}|^3+\max_{1\leq j\leq p}|Y_{ij}|^3+\max_{1\leq j\leq p}|W_{ij}|^3
\]
together with the inequality (\ref{eq: simp3}) in Lemma \ref{lem: correlations}, we conclude that the integral in (\ref{eq: integral}) is bounded from above up to a universal constant by
\[
\Ep\left[\max_{1\leq j\leq p}|X_{ij}|^31\left\{\max_{1\leq j\leq p}|X_{ij}|>\mathcal{T}\right\}\right]+\Ep\left[\max_{1\leq j\leq p}|Y_{ij}|^31\left\{\max_{1\leq j\leq p}|Y_{ij}|>\mathcal{T}\right\}\right]
\]
since $W_i$'s have the same distribution as that of $Y_i$'s. Therefore,
\[
|III_2|\lesssim (M_{n,X}(\phi)+M_{n,Y}(\phi))\phi\beta^2/n^{1/2}=M_n(\phi)\phi\beta^2/n^{1/2}.
\]

To bound $III_1$, recall the definition of $h(w,t)$ in (\ref{eq: def iota}). Note that $m_{j k l}(Z^{(i)}+\tau Z_i)=0$ for all $\tau\in[0,1]$ whenever $h(Z^{(i)},1)=0$ and $\chi_i=1$, so that 
\begin{equation}\label{eq: switching cost}
\chi_i |m_{j k l}(Z^{(i)}+\tau Z_i)| = h(Z^{(i)},1) \chi_i |m_{j k l}(Z^{(i)}+\tau Z_i)|.
\end{equation}
Indeed if $\chi_i=1$, then $\max_{1\leq j\leq p}|Z_{i j}|\leq 3/(4\beta)< 1/\beta$, and so when $h(Z^{(i)},1)=0$ and $\chi_i=1$, we have  $h(Z^{(i)}+\tau Z_i,0)=0$, which in turn implies that either $F_{\beta}(Z^{(i)}+\tau Z_i)\leq 0$ or $F_{\beta}(Z^{(i)}+\tau Z_i)\geq \phi^{-1}$ because of (\ref{eq: F properties}); in both cases, the assertion follows from the definitions of $m$ and $g$. Hence
\begin{align}
|III_1|&\leq \sum_{j,k,l=1}^p\sum_{i=1}^n\int_0^1\int_0^1\Ep[\chi_i|m_{jkl}(Z^{(i)}+\tau Z_i)\dot{Z}_{ij}Z_{ik}Z_{il}|]d\tau dt \notag \\
&\lesssim \sum_{j,k,l=1}^p\sum_{i=1}^n\int_0^1\int_0^1\Ep[\chi_i h(Z^{(i)},1) U_{j k l}(Z^{(i)}+\tau Z_i)|\dot{Z}_{ij}Z_{ik}Z_{il}|]d\tau dt \notag \\
&\lesssim \sum_{j,k,l=1}^p\sum_{i=1}^n\int_0^1\int_0^1\Ep[\chi_i h(Z^{(i)},1) U_{j k l}(Z^{(i)})|\dot{Z}_{ij}Z_{ik}Z_{il}|]d\tau dt \notag \\
&\lesssim \sum_{j,k,l=1}^p\sum_{i=1}^n\int_0^1\Ep[h(Z^{(i)},1) U_{jkl}(Z^{(i)})]\Ep[|\dot{Z}_{ij}Z_{ik}Z_{il}|]d t, \label{eq: integral 2}
\end{align}
where the second inequality follows from (\ref{eq: property1}) and \eqref{eq: switching cost}, the third inequality from (\ref{eq: property3}), and the fourth inequality from the indepence of $Z^{(i)}$ from $\dot{Z}_{ij}Z_{ik}Z_{il}$.
Then we split the integral in (\ref{eq: integral 2}) by inserting $\chi_i+(1-\chi_i)$ under the first expectation sign. We have
\begin{align*}
&\sum_{j,k,l=1}^p\sum_{i=1}^n\int_0^1\Ep[(1-\chi_i)h(Z^{(i)},1) U_{jkl}(Z^{(i)})]\Ep[|\dot{Z}_{ij}Z_{i k}Z_{il}|] dt \\
&\quad \lesssim \phi\beta^2\sum_{i=1}^n\int_0^1 \Ep[1-\chi_i]\Ep\left[\max_{1\leq j,k,l\leq p}|\dot{Z}_{ij}Z_{ik}Z_{i l}|\right] dt \lesssim M_n(\phi)\phi\beta^2/n^{1/2},
\end{align*}
where the last inequality follows from the argument similar to that used to bound $III_2$ with applying (\ref{eq: simp1}) and (\ref{eq: simp2}) instead of (\ref{eq: simp3}) in Lemma \ref{lem: correlations}. Moreover, since $h(Z^{(i)},1)=0$ whenever $h(Z,2)=0$ and $\chi_i=1$ (which follows from the same argument as before), so that
$$
\chi_i h(Z^{(i)},1) = \chi_i h(Z^{(i)},1) h(Z,2),
$$
we have
\begin{align}
&\sum_{j,k,l=1}^p\sum_{i=1}^n\int_0^1\Ep[\chi_ih(Z^{(i)},1) U_{jkl}(Z^{(i)})]\Ep[|\dot{Z}_{ij}Z_{ik}Z_{il}|]d t\notag \\
&\qquad \lesssim \sum_{j,k,l=1}^p\sum_{i=1}^n\int_0^1\Ep[\chi_ih(Z^{(i)},1) U_{jkl}(Z)]\Ep[|\dot{Z}_{ij}Z_{ik}Z_{il}|]d t\notag \\
&\qquad \lesssim \sum_{j,k,l=1}^p\sum_{i=1}^n\int_0^1\Ep[h(Z,2) U_{jkl}(Z)]\Ep[|\dot{Z}_{ij}Z_{ik}Z_{il}|]d t\notag \\ 
&\qquad = \sum_{j,k,l=1}^p\int_0^1\Ep[h(Z,2) U_{j k l}(Z)]\sum_{i=1}^n\Ep[|\dot{Z}_{ij}Z_{ik}Z_{il}|]d t\notag \\
&\qquad \lesssim \phi\beta^2\int_0^1\Ep[h(Z,2)]\max_{1\leq j,k,l\leq p}\sum_{i=1}^n \Ep[|\dot{Z}_{ij}Z_{ik}Z_{il}|]d t. \label{eq: integral 3}
%&\qquad =(L_1\psi\beta^2/n^{1/2})\int_0^1\Ep[\tilde{g}(Z,2)]d t.
\end{align}
To bound (\ref{eq: integral 3}), observe that
\begin{multline*}
|\dot{Z}_{ij}Z_{ik}Z_{il}|\lesssim \frac{\omega(t)}{n^{3/2}}\Big(|X_{ij}|^{3}+|Y_{ij}|^{3}+|W_{ij}|^{3}\\ + |X_{ik}|^{3}+|Y_{ik}|^{3}+|W_{ik}|^{3} + |X_{il}|^{3}+|Y_{il}|^{3}+|W_{il}|^{3}\Big),
\end{multline*}
which, together with the facts that $\Ep[|W_{ij}|^3] = \Ep[|Y_{ij}|^3]$ and $\Ep[|Y_{ij}|^3]\lesssim (\Ep[|Y_{i j}|^2])^{3/2}=(\Ep[|X_{ij}|^2])^{3/2}\leq \Ep[|X_{ij}|^3]$, implies that
\[
\max_{1\leq j,k,l\leq p}\sum_{i=1}^{n}\Ep[|\dot{Z}_{ij}Z_{ik}Z_{il}|]  \lesssim  \frac{\omega(t)}{n^{3/2}}\max_{1\leq j \leq p}\sum_{i=1}^{n}(\Ep[|X_{ij}|^{3}]+\Ep[|Y_{ij}|^{3}]) \lesssim \frac{\omega(t)}{n^{1/2}}L_{n}.
\]
Meanwhile, observe that
\[
\Ep[h(Z,2)]=\Pr(Z \leq \overline{I})-\Pr(Z \leq \underline{I}),
\]
where
\begin{align*}
Z & =  \frac{1}{\sqrt{n}}\sum_{i=1}^{n}(\sqrt{tv}X_{i}+\sqrt{t(1-v)}Y_{i}+\sqrt{1-t}W_{i})\\
 & \stackrel{d}{=}  \frac{1}{\sqrt{n}}\sum_{i=1}^{n}(\sqrt{tv} X_{i}+\sqrt{1-tv}Y_{i}),
\end{align*}
and $\underline{I}=y-\phi^{-1}-2\beta^{-1}$, $\overline{I}=y+\phi^{-1}+2\beta^{-1}$; here the notation $\stackrel{d}{=}$ denotes equality in distribution, and $\underline{I}$ and $\overline{I}$ are vectors in $\mathbb{R}^p$ (recall
the rules of summation of vectors and scalars defined in Section 1.1).
Now by the definition of $\varrho_{n}$,
\begin{align*}
\Pr(Z\leq\overline{I}) & \leq  \Pr(S_{n}^{Y}\leq\overline{I})+\varrho_{n},\quad \Pr(Z\leq\underline{I}) \geq  \Pr(S_{n}^{Y}\leq\underline{I})-\varrho_{n},
\end{align*}
and by Lemma \ref{lem: anti-concentration},
\[
\Pr(S_{n}^{Y}\leq\overline{I})-\Pr(S_{n}^{Y}\leq\underline{I})\lesssim\phi^{-1}\log^{1/2}p
\]
since $\beta^{-1}\lesssim \phi^{-1}$ and $\Ep[(S_{n j}^Y)^2] = \Ep[(S_{n j}^X)^2] = n^{-1}\sum_{i=1}^n \Ep[X_{i j}^2]\geq b$ for all $j=1,\dots,p$.
Hence
\[
\Ep[h(Z,2)]\lesssim\varrho_{n}+\phi^{-1}\log^{1/2}p.
\]
By these bounds, together with the fact that $\int_{0}^{1}\omega(t)dt\lesssim1$, we conclude that
\[
(\ref{eq: integral 3})\lesssim\frac{\phi\beta^{2}L_{n}}{n^{1/2}}(\varrho_{n}+\phi^{-1}\log^{1/2}p)\lesssim\frac{\phi^{2}\log^{2}p}{n^{1/2}}(\phi L_{n}\varrho_{n}+L_{n}\log^{1/2}p),
\]
where we have used $\beta=\phi \log p$. The desired assertion (\ref{eq: to prove lemma 5.1}) then follows.

\medskip

{\bf Step 2}. We are now in position to finish the proof. Let
\[
V_{n}:=\sqrt{v}S^X_{n}+\sqrt{1-v}S^Y_{n}.
\]
Then we have
\begin{align*}
&\Pr(V_{n}\leq y-\phi^{-1})
\leq \Pr(F_\beta(V_{n})\leq 0) \leq \Ep[m(V_{n})] \\
&\quad \leq \Pr(F_\beta(S_{n}^W)\leq\phi^{-1})+(\Ep[m(V_{n})]-\Ep[m(S_{n}^W)]) \\
&\quad \leq \Pr(S_{n}^W \leq y+\phi^{-1})+|\Ep[\mathcal{I}_n]| \\
&\quad \leq \Pr(S_{n}^W \leq y-\phi^{-1})+C\phi^{-1}\log^{1/2} p+|\Ep[\mathcal{I}_n]|,
\end{align*}
where the first three lines follow from the properties of $F_{\beta}(w)$ and $g(t)$ (recall that $m(w)=g(F_\beta(w))$), and the last inequality follows from Lemma \ref{lem: anti-concentration}. Here the constant $C$ depends only on $b$. 
Likewise we have
\[
\Pr(V_{n} \leq y-\phi^{-1}) \geq \Pr(S_{n}^W \leq y-\phi^{-1})-C\phi^{-1}\log^{1/2} p - |\Ep[\mathcal{I}_n]|.
\]
The conclusion of the lemma follows from combining these inequalities with the bound on $|\Ep[\mathcal{I}_{n}]|$ derived in Step 1.
\qed

\subsection*{Proof of Corollary \ref{cor: induction lemma}}
Pick any hyperrectangle 
\[
A=\{ w \in \R^p: w_j \in [a_j,b_j] \ \text{for all} \ j=1,\dots,p\}.
\]
For $i=1,\dots,n$, consider the random vectors $\tilde{X}_i$ and $\tilde{Y}_i$ in $\R^{2p}$ defined by $\tilde{X}_{ij}=X_{ij}$ and $\tilde{Y}_{ij}=Y_{ij}$ for $j=1,\dots,p$, and $\tilde{X}_{ij}=-X_{i,j-p}$ and $\tilde{Y}_{ij}=-Y_{i,j-p}$ for $j=p+1,\dots,2p$. 
Then
\[
\Pr(S_{n}^X \in A)=\Pr(S_{n}^{\tilde{X}} \leq y), \quad \Pr(S_{n}^Y\in A)=\Pr(S_{n}^{\tilde{Y}} \leq y),
\]
where the vector $y \in \R^{2p}$ is defined by $y_j=b_j$ for $j=1,\dots,p$ and $y_j=-a_{j-p}$ for $j=p+1,\dots,2p$, and $S_{n}^{\tilde{X}}$ and $S_{n}^{\tilde{Y}}$ are defined as $S_{n}^X$ and $S_{n}^Y$ with $X_i$'s and $Y_i$'s replaced by $\tilde{X}_i$'s and $\tilde{Y}_i$'s. Hence the corollary follows from applying Lemma \ref{lem: induction lemma} to  $\tilde{X}_{1},\dots,\tilde{X}_{n}$ and $\tilde{Y}_{1},\dots,\tilde{Y}_{n}$. 
% Note that the term $M_n(\phi)$ in the lemma is replaced by $M_n(2\phi)$ in the corollary because
% \begin{align*}
% &\Ep\left[\max_{1\leq j\leq p}|X_{ij}|^31\left\{\max_{1\leq j\leq p}|X_{ij}|>\sqrt{n}/(4\phi\log(2 p))\right\}\right]\\
% &\quad \leq\Ep\left[\max_{1\leq j\leq p}|X_{ij}|^31\left\{\max_{1\leq j\leq p}|X_{ij}|>\sqrt{n}/(8\phi\log( p))\right\}\right]
% \end{align*}
% since we assume that $p\geq 2$.
\qed

\section{Proofs for Section \ref{sec: main results}}

\subsection*{Proof of Theorem \ref{thm: main}}
The proof relies on Lemma \ref{lem: induction lemma} and its Corollary \ref{cor: induction lemma}.
Let $K'$ denote a constant from the conclusion of Corollary \ref{cor: induction lemma}. This constant depends only on $b$. Set $K_2:=1/(K'\vee 1)$ in (\ref{eq: psi}), so that
\[
\phi_n=\frac{1}{K'\vee 1}\left(\frac{\overline{L}_{n}^2\log^4 p}{n}\right)^{-1/6}.
\]
Without loss of generality, we may assume that $\phi_n \geq 2$; otherwise, the assertion of the theorem holds trivially by setting $K_1=2(K'\vee 1)$. 

Then applying Corollary \ref{cor: induction lemma} with $\phi=\phi_n/2$, we have
\[
\varrho_{n}'\leq \frac{\varrho_{n}'}{8(K'\vee 1)^2}+\frac{3(K'\vee 1)^2 \overline{L}_{n}^{1/3}\log^{7/6} p}{n^{1/6}}+\frac{M_{n}(\phi_n)}{8(K'\vee 1)^2\overline{L}_{n}}.
\]
Since $8(K'\vee 1)^2 >1$, solving this inequality for $\varrho_n'$ and observing that $\rho_n(\mA^{\re})\leq \varrho_{n}'$ leads to the desired assertion.
\qed

\medskip

Before proving Proposition \ref{cor: examples}, we shall verify the following elementary inequality. 
\begin{lemma}
\label{lem: truncated moment}
Let $\xi$ be a non-negative random variable such that $\Pr(\xi >x)\leq A e^{-x/B}$ for all $x\geq 0$ and for some constants $A, B>0$. Then for every $t\geq 0$, $\Ep[\xi ^31\{\xi >t\}]\leq 6A(t+B)^3e^{-t/B}$. 
\end{lemma}
\begin{proof}[Proof of Lemma \ref{lem: truncated moment}]
Observe that 
\begin{align*}
\Ep[\xi^31\{\xi>t\}]
%&\leq \int_0^\infty \Pr(X^31\{X>t\}>s)d s\\
%&=\int_0^\infty \Pr(X 1\{X>t\}>s^{1/3})d s\\
%&=3\int_0^\infty \Pr(X 1\{X>t\}>s)s^2 d s\\
&=3\int_0^t \Pr(\xi >t)x^2 d x + 3\int_t^\infty \Pr(\xi>x)x^2 d x \\
&=\Pr(\xi>t)t^3 + 3\int_t^\infty \Pr(\xi>x)x^2 d x.
\end{align*}
Since $\Pr(\xi>x)\leq A e^{-x/B}$, using integration by parts, we have 
\begin{align*}
\int_t^\infty \Pr(\xi>s)x^2 d x
%&=A\int_t^\infty  e^{-s/B}s^2d s\\
%&=A B t^2 e^{-t/B}+2 A B\int_t^\infty e^{-s/B}s d s\\
%&=A(B t^2+ 2B^2 t)e^{-t/B}+2 A B^2\int_t^\infty e^{-s/B}d s\\
\leq A(B t^2+ 2B^2 t + 2B^3)e^{-t/B},
\end{align*}
which leads to 
\[
\Ep[\xi^31\{\xi>t\}]\leq A(t^3+3B t^2+ 6B^2 t + 6B^3)e^{-t/B}\leq 6A(t+B)^3e^{-t/B},
\]
completing the proof.
\end{proof}

\subsection*{Proof of Proposition \ref{cor: examples}}
The proof relies on application of Theorem \ref{thm: main}. Without loss of generality, we may assume that
\begin{equation}
\label{eq: additional assumption}
\frac{B_n^2\log^{7} (pn)}{n}\leq c:=\min \{ (c_{1}/2)^3, (K_2/2)^6 \},
% \min\left(b/(6C_{1}),(c_{1}/2)^3, (K_2/2)^6\right),
\end{equation}
where $K_2$ appears in (\ref{eq: psi}) and $c_{1} > 0$ is a constant that depends only on $b$ ($c_{1}$ will be defined later), since
otherwise we can make the assertions trivial by setting $C$ large enough. 

% \textbf{Step 1}. We begin with verifying that
% \[
% \frac{1}{m}\sum_{i=1}^m \Ep[X_{i j}^2]\geq b/2, \ \text{for all} \ m=[n-\log n-1],\dots,n, j=1,\dots,p.
% \]
% Indeed, under either (E.1) or (E.2), we have $\Ep[X_{ij}^2]\leq C_{1}B_n^2$ for all $i=1,\dots,n$ and $j=1,\dots,p$ for some universal constant $C_{1}$. Hence as $n \geq 4$, 
% \[
% \frac{1}{n}\sum_{i=m+1}^n\Ep[X_{ij}^2] \leq n^{-1}C_{1}B_n^2(2+\log n) \leq 3 n^{-1} C_{1}B_n^2\log n \leq 3c C_{1}.
% \]
% Since $c>0$ is such that $3cC_{1}\leq b/2$, we have
% \[
% \frac{1}{m}\sum_{i=1}^m\Ep[X_{ij}^2]\geq \frac{1}{n}\sum_{i=1}^m\Ep[X_{ij}^2]\geq \frac{1}{n}\sum_{i=1}^n\Ep[X_{ij}^2]-b/2\geq b/2.
% \]

Now by Theorem \ref{thm: main}, we have
\[
\rho_n(\mA^{\re})\leq K_1\left[\left(\frac{\overline{L}_{n}^2\log^{7}p}{n}\right)^{1/6}+\frac{M_{n,X}(\phi_n)+M_{n,Y}(\phi_n)}{\overline{L}_{n}}\right],
\]
where $\phi_{n} = K_{2} \{ n^{-1} \overline{L}_{n}^{2} \log^{4}p\}^{-1/6}$, and $\overline{L}_{n}$ is any constant such that $\overline{L}_{n} \geq L_n$.
Recall that 
\begin{align*}
&L_n = \max_{1\leq j\leq p}\sum_{i=1}^n \Ep[|X_{ij}|^3]/n, \\
&M_{n,X}(\phi_n) = n^{-1}\sum_{i=1}^n\Ep\left[\max_{1\leq j\leq p}|X_{ij}|^31\left\{\max_{1\leq j\leq p}|X_{ij}|>\sqrt{n}/(4\phi_n\log p)\right\}\right],
\end{align*}
and $M_{n,Y}(\phi_n)$ is defined similarly with $X_{i j}$'s replaced by $Y_{i j}$'s.

It remains to choose a suitable constant $\overline{L}_n$ such that $\overline{L}_{n} \geq L_n$ and bound  $M_{n,X}(\phi_n)$ and $M_{n,Y}(\phi_n)$. 
To this end, we consider cases (E.1) and (E.2) separately. In what follows, the notation $\lesssim$ means that the left-hand side is bounded by the right-hand side up to a positive constant that depends only on $b$ under case (E.1), and on $b$ and $q$ under case (E.2).

\medskip

\textbf{Case (E.1)}.
Set $\overline{L}_n:=B_n$. By condition (M.2), we have $L_n \leq B_n =\overline{L}_{n}$.
Observe that (E.1) implies that $\|X_{ij}\|_{\psi_1}\leq B_n$ for all $i$ and $j$. In addition, since each $Y_{i j}$ is Gaussian and $\Ep[Y_{ij}^2]=\Ep[X_{ij}^2]$, $\|Y_{ij}\|_{\psi_1} \leq C_1B_n$ for all $i$ and $j$ and some universal constant $C_1>0$.  Hence by Lemma 2.2.2 in \cite{VW96}, we have for some universal constant $C_{2}> 0$, $\|\max_{1\leq j\leq p} X_{ij}\|_{\psi_1} \leq C_{2} B_n\log p$ and $\|\max_{1\leq j\leq p} Y_{ij}\|_{\psi_1} \leq C_{2} B_n\log p$.
Together with Markov's inequality, this implies that for every $t > 0$,
\[
\Pr\left(\max_{1\leq j\leq p}|X_{ij}|> t\right)\leq 2\exp\left(-\frac{t}{C_{2} B_n\log p}\right).
\]
Applying Lemma \ref{lem: truncated moment}, we have
\[
M_{n,X}(\phi_n) \lesssim (\sqrt{n}/(\phi_n\log p)+B_n\log p)^3\exp\left(-\frac{\sqrt{n}}{4C_{2}\phi_n B_n\log^{2} p}\right).
\]
Here 
\begin{align*}
&\frac{\sqrt{n}}{4C_{2}\phi_n B_n\log^{2} p} = \frac{c_{1}n^{1/3}}{B_{n}^{2/3} \log^{4/3}p} \quad \left (c_{1} := \frac{1}{4K_{2}C_{2}}\right) \\
&\quad \geq c_{1} c^{-1/3} \log (pn) \geq 2 \log (pn). \quad (\text{by (\ref{eq: additional assumption})}).
\end{align*}
Moreover, by (\ref{eq: additional assumption}) and $\phi_{n}^{-1} = K_{2}^{-1} \{ n^{-1} B_{n}^{2} \log^{4}p \}^{1/6} \leq c^{1/6}/K_{2} \leq 1$,
we have $(\sqrt{n}/(\phi_n\log p)+B_n\log p)^3 \lesssim n^{3/2}$, 
which implies that 
\[
M_{n,X}(\phi_{n}) \lesssim n^{3/2} \exp (-2\log (pn)) \leq n^{-1/2}.
\]
The same reasoning also gives $M_{n,Y}(\phi_{n}) \lesssim n^{-1/2}$.
%For $M_{n,Y}(\phi_{n})$, since $\Ep[Y_{ij}^2]=\Ep[X_{ij}^2] \leq C_{1}B_n^2$ and hence $\|Y_{ij}\|_{\psi_1} \lesssim B_n$ for all $i$ and $j$ (as each $Y_{ij}$ is Gaussian),
%we also have $M_{n,Y}(\phi_{n}) \lesssim n^{-1/2}$. 
The conclusion of the proposition in this case now follows from the fact that $n^{-1/2} B_{n}^{-1} \leq D_{n}^{(1)}$.
%  Therefore, $M_{n,Y}(\phi_n)\lesssim n^{-1/2}$ by the same argument as that used above, and so
% $$
% \frac{M_{n,X}(\phi_n)+M_{n,Y}(\phi_n)}{\overline{L}_{n}} \leq M_{n,X}(\phi_n)+M_{n,Y}(\phi_n) \lesssim  \left(\frac{B_n^2(\log (p n))^7}{n}\right)^{1/6}
% $$
% where we used $B_n\geq 1$.
%  Inserting this bound and $\bar L_n$ into (\ref{main bound}) gives the required claim (\ref{eq: rate example 1}).

\medskip
\textbf{Case  (E.2)}. Without loss of generality, in addition to (\ref{eq: additional assumption}), we may assume that
\begin{equation}\label{eq: additional assumption2}
\frac{B_n \log^{3/2} p}{n^{1/2-1/q}} \leq (K_2/2)^{3/2}.
\end{equation}
Set
\[
\overline{L}_n : = \left \{ B_n+\frac{B_n^2}{n^{1/2-2/q} \log^{1/2}p}\right \}.
\]
Then $L_n\leq B_n\leq \overline{L}_{n}$. 
As the map $x \mapsto x^{1/3}$ is sub-linear, $\{ n^{-1} \overline{L}_{n}^{2} \log^{7}p \}^{1/6} \leq D_{n}^{(1)}+D_{n,q}^{(2)} \leq K_{2}$, so that $\phi_{n}^{-1} = K_{2}^{-1} \{ n^{-1} \overline{L}_{n}^{2} \log^{4} p \}^{1/6}  \leq 1$. 
% \[
% \frac{1}{\phi_n}= \frac{1}{K_2}\left(\frac{\overline{L}_{n}^2(\log p)^4}{n}\right)^{1/6} \leq \frac{1}{K_2} \left(\frac{\overline{L}_{n}^2(\log p)^7}{n}\right)^{1/6} \leq 1.
% \]

Note that for any real-valued random variable $Z$ and any $t > 0$, $\Ep [ |Z|^3 1(|Z|>t) ]\leq \Ep[ |Z|^3 (|Z|/t)^{q-3} 1(|Z|>t) ]\leq t^{3-q} \Ep[|Z|^{q}]$. Hence
\begin{equation*}
M_{n,X}(\phi_n)\lesssim \frac{B_n^q\phi_{n}^{q-3} \log^{q-3}p}{n^{q/2-3/2}}.
\end{equation*}
Here using the bound $\overline{L}_{n}^{-1} \leq B_{n}^{-2} n^{1/2-2/q} \log^{1/2}p$, we have that $\phi_{n} \lesssim n^{1/3-2/(3q)} B_{n}^{-2/3} (\log p)^{-1/2}$,
so that 
% \[
% \phi_n \lesssim \frac{n^{1/3-2/(3q)}}{B_n^{2/3}(\log p)^{1/2}}
% \]
% which follows from (\ref{eq: L inverse bound}), gives us
\[
M_{n,X}(\phi_n)\lesssim \frac{B_n^{q/3+2}(\log p)^{q/2-3/2}}{n^{q/6+1/6-2/q}},
\]
which implies that 
\begin{align*}
M_{n,X}(\phi_{n})/\overline{L}_{n} 
&\lesssim \frac{B_n^{q/3+2}(\log p)^{q/2-3/2}}{n^{q/6+1/6-2/q}}\cdot  \frac{n^{1/2 - 2/q} \log^{1/2}p}{B_n^2}\\
&\lesssim\frac{1}{\log p}\Big(\frac{B_n^2 \log^3 p}{n^{1 - 2/q}}\Big)^{q/6} \lesssim D_{n,q}^{(2)}.
\end{align*}
Meanwhile, as in the previous case, we have $M_{n,Y}(\phi_{n}) \lesssim n^{-1/2}$, which leads to the desired conclusion in this case.
\qed

\section{Proofs for Section \ref{sec: friendly sets}}

\subsection*{Proof of Proposition \ref{cor: sparsely convex bodies}}
Here $C$ denotes a generic positive constant that depends only on $a,b$, and $d$ if (E.1$'$) is satisfied, and on $a,b,d$, and $q$ if (E.2$'$) is satisfied; the value of $C$ may change from place to place.  Pick any $A\in\mathcal{A}\subset\mA^{\si}(a,d)$. Let $A^m=A^m(A)$ be an approximating $m$-generated convex set as in condition (C). 
By assumption, $A^m\subset A\subset A^{m,\epsilon}$, so that 
by letting
\[
\overline{\rho} :=|\Pr(S_n^X\in A^{m})-\Pr(S_n^Y\in A^{m})|\vee |\Pr(S_n^X\in A^{m,\epsilon})-\Pr(S_n^Y\in A^{m,\epsilon})|,
\]
we have $\Pr(S_n^X \in A) \leq \Pr(S_n^X \in A^{m,\epsilon}) \leq \Pr(S_n^Y \in A^{m,\epsilon})+\overline{\rho}$. 
Here observe that $(v'S_{n}^{Y})_{v \in \mathcal{V}(A^{m})}$ is a Gaussian random vector with dimension $\Card (\mathcal{V}(A^{m})) = m \leq (pn)^{d}$ such that, by condition (M.1$'$),  the variance of each coordinate is bounded from below by $b$. 
Hence by Lemma \ref{lem: anti-concentration}, we have 
\begin{align*}
\Pr(S_n^Y \in A^{m,\epsilon}) &= \Pr \{ v'S_{n}^{Y} \leq \mathcal{S}_{A^{m}}(v)+\epsilon \ \text{for all} \ v \in \mathcal{V}(A^{m}) \} \\
&\leq \Pr \{ v'S_{n}^{Y} \leq \mathcal{S}_{A^{m}}(v) \ \text{for all} \ v \in \mathcal{V}(A^{m}) \} + C \epsilon \log^{1/2}m \\
&= \Pr(S_n^Y \in A^{m}) + C\epsilon \log^{1/2} (pn),
\end{align*}
so that 
\begin{align*}
\Pr(S_n^X \in A) &\leq \Pr(S_n^Y\in A^m)+C\epsilon \log^{1/2}(pn) +\overline{\rho} \\
&\leq \Pr(S_n^Y\in A)+C\epsilon \log^{1/2}(pn) + \overline{\rho}. \quad (\text{by} \ A^{m} \subset A)
\end{align*}
Likewise we have $\Pr(S_n^X \in A) \geq \Pr(S_n^Y \in A)-C\epsilon \log^{1/2}(pn)-\overline{\rho}$, 
by which we conclude
\[
|\Pr(S_n^X \in A) - \Pr(S_n^Y \in A)| \leq C\epsilon \log^{1/2}(pn) + \overline{\rho}. 
\]
Recalling that $\epsilon = a/n$ and $B_{n} \geq 1$, we have $\epsilon \log^{1/2}(pn) \leq C D_{n}^{(1)}$. 
Hence the assertions of the proposition follow if we prove 
\[
\overline{\rho} \leq 
\begin{cases}
C D_{n}^{(1)} & \text{if (E.1$'$) is satisfied}, \\
C\{ D_{n}^{(1)}+D_{n,q}^{(2)} \} & \text{if (E.2$'$) is satisfied}.
\end{cases}
\]
However, this follows from application of Proposition \ref{cor: examples} to $\tilde{X}_1,\dots,\tilde{X}_n$ instead of $X_{1},\dots,X_{n}$. 
% under (E.1$'$) by Corollary \ref{cor: examples} applied to vectors $\tilde{X}_1,\dots,\tilde{X}_n$ and $\tilde{Y}_1,\dots,\tilde{Y}_n$ where the latter sequence is defined by analogy with the former sequence but using vectors $Y_1,\dots,Y_n$. The asserted claim in the (E.2$'$) case follows from the same argument. This completes the proof of the corollary.
\qed

\subsection*{Proof of Corollary \ref{cor: log-concave distributions}}
Since $X_{i}$ is a centered random vector with a log-concave distribution in $\R^{p}$, Borell's inequality \citep[see][Lemma 3.1]{B74} implies that $\| v'X_{i} \|_{\psi_{1}} \leq c (\Ep[ (v'X_{i})^{2} ])^{1/2}$ for all $v \in \R^{p}$ for some universal constant $c > 0$ \citep[see][Appendix III]{MS86}; hence if the maximal eigenvalue of each $\Ep [X_{i}X_{i}']$ is bounded by a constant $k_2$, then every simple convex set $A\in\mA^{\si}(a,d)$ obeys conditions (M.2$'$) and (E.1$'$) with $B_n$ replaced by a constant that depends only on $c$ and $k_2$. Besides if the minimal eigenvalue of each $\Ep [X_{i}X_{i}']$ is bounded from below by a constant $k_1$, then every simple convex set $A\in\mA^{\si}(a,d)$ obeys condition (M.1$'$) with $b$ replaced by a positive constant that depends only on $k_1$. Hence the conclusion of the corollary follows from application of Proposition \ref{cor: sparsely convex bodies}.
\qed

\subsection*{Proof of Proposition \ref{cor: sparsely convex sets}}
Here $C$ denotes a positive constant that depends only on $b$ and $s$ if condition (E.1) is satisfied, and on $b,s$, and $q$ if condition (E.2) is satisfied; the value of $C$ may change from place to place. Without loss of generality,  we may assume that $B_n^2 \leq n$ since otherwise the assertions are trivial.

Let $R := p n^{5/2}$ and $V^R := \{w\in \mathbb{R}^p: \max_{1\leq j\leq p}|w_j|>R\}$. Fix any $A\in\mA^{\spa}(s)$. Then $A=\check{A} \cup (A \cap V^R)$ for some $s$-sparsely convex set $\check{A} \subset \R^p$ such that $\sup_{w\in \check{A}}\max_{1\leq j\leq p}|w_j|\leq R$. Now observe that by Markov's inequality,
\begin{align*}
&\Pr\left(\max_{i,j}|X_{ij}|>p n^2\right)
\leq \frac{\Ep[\max_{i,j}|X_{ij}|]}{p n^2} \leq \frac{\Ep[\sum_{i,j}|X_{ij}|]}{p n^2} \\
&\quad \leq \max_{i,j}\Ep[|X_{ij}|]/n\leq C B_n/n\leq C/n^{1/2},
\end{align*}
where $\max_{i,j}$ stands for $\max_{1\leq i\leq n}\max_{1\leq j\leq p}$. Hence
\[
\Pr\left(S_{n}^X\in V^R\right)\leq C/n^{1/2},
\]
and similarly,
\[
\Pr\left(S_{n}^Y\in V^R\right)\leq C/n^{1/2}.
\]
So,
\[
|\Pr(S_n^X\in A)-\Pr(S_n^Y\in A)| \leq |\Pr(S_n^X\in \check{A})-\Pr(S_n^Y\in \check{A})|+ C/n^{1/2}.
\]
Therefore it suffices to consider the case where the sets $A \in \mA^{\spa}(s)$ are such that 
\begin{equation}\label{eq: boundedness property}
\sup_{w\in A}\max_{1\leq j\leq p}|w_j|\leq R.
\end{equation}
Further, let $\varepsilon=n^{-1}$, and define $\mA^{\spa}_1(s)$ as the class of all sets $A\in \mA^{\spa}(s)$ satisfying (\ref{eq: boundedness property}) and containing a ball with radius $\varepsilon$ and center at, say, $w_A$. Also define $\mA^{\spa}_2(s)$ as the class of all sets $A\subset \mA^{\spa}(s)$ satisfying (\ref{eq: boundedness property}) and containing no ball of radius $\varepsilon$. We bound $\rho_n(\mA^{\spa}_1(s))$ and $\rho_n(\mA^{\spa}_2(s))$ separately in two steps. In both cases, we rely on the following lemma,  whose proof is given after the proof of this proposition.

%the sets $A\in\mA^{\spa}(s)$ satisfying 
%let $\mA^{\spa}_1(s)$ denote the subclass of $\mA^{\spa}(s)$ consisting of all sets $A$ in $\mA^{\spa}(s)$ satisfying $\sup_{w\in A}\max_{1\leq j\leq p}|w_j|\leq R$ and containing a ball with radius $\varepsilon$ and center at, say, $w_A$. Let $\mA^{\spa}_2(s)=\mA^{\spa}(s) \setminus \mA^{\spa}_1(s)$. 

%We divide the proof into five steps. In Steps 1-4, we verify conditions (C), (M.1$'$), (M.2$'$), (E.1$'$) (if (E.1) is satisfied), and condition (E.2$'$) (if (E.2) is satisfied) for all $A\in\mA^{\spa}_1(s)$. An application of Proposition \ref{cor: sparsely convex bodies} then shows that the assertions (\ref{eq: sparsely convex sets with mgf}) and (\ref{eq: sparsely convex sets with finite moments}) hold with $\rho_n(\mA^{\spa}(s))$ replaced by $\rho_n(\mA^{\spa}_1(s))$. Step 5 shows that the same assertions also hold with $\rho_n(\mA^{\spa}(s))$ replaced by $\rho_n(\mA^{\spa}_2(s))$. Since $\rho_n(\mA^{\spa}(s))=\rho_n(\mA^{\spa}_1(s))\vee \rho_n(\mA^{\spa}_2(s))$, this will complete the proof. Step 1 relies on the following lemma, whose proof is given after the proof of this proposition.

\begin{lemma}%[Sparsely convex sets are simple convex sets]
\label{lem: verifying condition C.j}
Let $A$ be an $s$-sparsely convex set with a sparse representation $A=\cap_{q=1}^Q A_q$ for some $Q\leq p^s$. Assume that $A$ contains the origin, that $\sup_{w\in A}\|w\|\leq R$, and that all sets $A_q$ satisfy $-A_q\subset\mu A_q$ for some $\mu\geq 1$. Then for any $\gamma>e/8$, there exists $\epsilon_0=\epsilon_0(\gamma)>0$ such that for any $0<\epsilon<\epsilon_0$, the set $A$ admits an approximation with precision $R \epsilon$ by an $m$-generated convex set $A^m$ where
\[
m\leq Q\Big(\gamma\sqrt{\frac{\mu+1}{\epsilon}}\log\frac{1}{\epsilon}\Big)^{s^2}.
\]
Moreover, the set $A^m$ can be chosen to satisfy 
\begin{equation}\label{eq: useful sparsity condition}
\|v\|_0\leq s\text{ for all }v\in \mathcal{V}(A^m).
\end{equation}
Therefore, since $Q\leq p^s$, if $R\leq (p n)^{d_0}$ and $\mu\leq (p n)^{d_0}$ for some constant $d_0\geq 1$, then the set $A$ satisfies condition {\em (C)} with $a=1$ and $d$ depending only on $s$ and $d_0$, and the approximating $m$-generated convex set $A^m$ satisfying (\ref{eq: useful sparsity condition}).
\end{lemma}

{\bf Step 1}. 
Here we bound $\rho_n(\mA^{\spa}_1(s))$. Pick any $s$-sparsely convex set $A\in\mA^{\spa}_1(s)$ with a sparse representation $A=\cap_{q=1}^Q A_q$ for some $Q\leq p^s$. Below we verify conditions (C), (M.1$'$), (M.2$'$), and (E.1$'$) (or (E.2$'$)) for this set $A$. Consider the set $B:=A-w_A:=\{ w \in \R^p:w+w_A\in A\}$. The set $B$ contains a ball with radius $\varepsilon$ and center at the origin, satisfies the inequality $\|w\|\leq 2p^{1/2}R$ for all $w\in B$, and has a sparse representation $B=\cap_{q=1}^Q B_q$ where $B_q=A_q-w_A$. Clearly, each $B_q$ satisfies $-B_q\subset \mu B_q$ with $\mu=2p^{1/2}R/\varepsilon=2p^{3/2} n^{7/2}$. Therefore, applying Lemma \ref{lem: verifying condition C.j} to the set $B$ and noting that $A=B+w_A$ and $Q\leq p^s$, we see that the set $A$ satisfies condition (C) with $a=1$ and $d$ depending only on $s$, and an approximating $m$-generated convex set $A^m$ such that $\|v\|_0\leq s$ for all $v\in\mathcal{V}(A^m)$.

Further, since we have $\|v\|_0\leq s$ for all $v\in\mathcal{V}(A^m)$, the fact that the set $A$ satisfies condition (M.1$'$) follows immediately from (M.1$''$).

Next, we verify that the set $A$ satisfies condition (M.2$'$). For $v \in \mathcal{V}(A^m)$, let $J(v)$ be the set consisting of positions of non-zero elements of $v$, so that $\Card (J(v)) \leq s$. 
Using the inequality $(\sum_{j\in J(v)}|a_j|)^{2+k} \leq s^{1+k}\sum_{j\in J(v)}|a_j|^{2+k}$ for $a=(a_1,\dots,a_p)' \in \R^p$ (which follows from H\"{o}lder's inequality), we have 
\begin{align*}
&n^{-1}\sum_{i=1}^n\Ep[|v' X_i|^{2+k}] \leq n^{-1}\sum_{i=1}^n\Ep\Big[\Big(\sum_{j\in J(v)}|X_{ij}|\Big)^{2+k}\Big]\\
&\quad \leq s^{1+k}n^{-1}\sum_{i=1}^n\Ep\Big[\sum_{j\in J(v)}|X_{ij}|^{2+k}\Big] \leq s^{2+k}B_n^k\leq (B_n')^k 
\end{align*}
for $k=1$ or $2$, where $B_n'=s^3 B_n$, so that the set $A$ satisfies condition (M.2$'$) with $B_n$ replaced by $s^3 B_n$. 

Finally, we verify that the set $A$ satisfies condition (E.1$'$) when (E.1) is satisfied, or (E.2$'$) when (E.2) is satisfied. 
When (E.1) is satisfied, we have $\|X_{ij}\|_{\psi_1}\leq B_n$, so that $\|v' X_i\|_{\psi_1 }\leq \sum_{j\in J(v)}\|X_{ij}\|_{\psi_1} \leq s B_n$ showing that the set $A$ satisfies (E.1$'$) with $B_n$ replaced by $s B_n$. 

When (E.2) is satisfied, as $\Ep[\max_{v\in\mathcal{V}(A^m)}|v' X_i|^q] \leq s^q\Ep[\max_{1\leq j\leq p}|X_{ij}|^q]$,
the set $A$ satisfies (E.2$'$) with $B_n$ replaced by $s B_n$.

Thus, all sets $A\in\mA^{\spa}_1(s)$ satisfy conditions (C), (M.1$'$), (M.2$'$), and (E.1$'$) (or (E.2$'$)), and so applying Proposition \ref{cor: sparsely convex bodies} shows that the assertions (\ref{eq: sparsely convex sets with mgf}) and (\ref{eq: sparsely convex sets with finite moments}) hold with $\rho_n(\mA^{\spa}(s))$ replaced by $\rho_n(\mA^{\spa}_1(s))$.

\medskip

{\bf Step 2}.
Here we bound $\rho_n(\mA^{\spa}_2(s))$. Fix any $s$-sparsely convex set $A\in\mA^{\spa}_2(s)$ with a sparse representation $A=\cap_{q=1}^QA_q$ for some $Q\leq p^s$. We consider two cases separately. First, suppose that at least one $A_q$ does not contain a ball with radius $\varepsilon$. Then under condition (M.1$''$), Lemma \ref{lem: Nazarov anti-concentration2} implies that $\Pr(S_n^Y\in A_q)\leq C\varepsilon= C/n$ (since the Hilbert-Schmidt norm is equal to the square-root of the sum of squares of the eigenvalues of the matrix, under our condition (M.1$''$), the constant $C$ in the bound $C\varepsilon$ above depends only on $b$ and $s$). In addition, under conditions (M.1$''$) and (M.2), the Berry-Esseen theorem \citep[see][Theorem 1.3]{G91} implies that 
\[
\left|\Pr(S_n^X\in A_q)-\Pr(S_n^Y\in A_q)\right|\leq CB_n/n^{1/2}.
\]
Since $A \subset A_q$, both $\Pr(S_n^X\in A)$ and $\Pr(S_n^Y\in A)$ are bounded from above by $CB_n/n^{1/2}$, and so is absolute value of their difference. This completes the proof in this case.

%Then applying Lemma XXX to this $A_q$ as if $A_q$ was a set in $\R^{s_q}$ where $s_q\leq s$ is the number of main components of $A_q$  shows that there exist $v\in \mathbb{S}^{p-1}$, and $c_1,c_2\in \R$ such that $\|v\|_0\leq s$, $c_1\leq c_2$, $c_2-c_1\leq \sqrt{s}\varepsilon \leq  n^{-1/2}$, and $A_q$ is contained in $A^0:=\{w\in\R^p:c_1\leq v'w\leq c_2\}$.  Since $A\subset A_q$, it follows that $A\subset  A^0$. Moreover, the set $A^0$ trivially satisfies condition (C) with $a=0$ and $d$ depending only on $s$ and by the same arguments as those used in Steps 2-4, it follows that the set $A^0$ also satisfies conditions (M.1$'$), (M.2$'$), (E.1$'$) (if (E.1) is satisfied), and (E.2$'$) (if (E.2) is satisfied). Therefore, applying Corollary \ref{cor: sparsely convex bodies} shows that
%$
%|\Pr(S_n^X\in A^0)-\Pr(S_n^Y\in A^0)|
%$
%is bounded from above by the quantities on the right-hand sides of (\ref{eq: best rate}) and (\ref{eq: best rate 2}) depending on whether (E.1) or (E.2) is satisfied.
%In addition, $\Pr(S_n^Y\in A^0)\leq C(\log p/n)^{1/2}$ by Lemma \ref{lem: anti-concentration} and condition (M.1$'$). This implies that both $\Pr(S_n^X\in A)$ and $\Pr(S_n^Y\in A)$ are bounded from above by the quantities on the right-hand sides of (\ref{eq: sparsely convex sets with mgf}) and (\ref{eq: sparsely convex sets with finite moments}) depending on whether (E.1) or (E.2) is satisfied, and so is their difference. This completes the proof in this case.

Second, suppose that each $A_q$ contains a ball with radius $\varepsilon$ (possibly depending on $q$). Then applying Lemma \ref{lem: verifying condition C.j} to each $A_q$ separately shows that for $m\leq (p n)^d$ with $d$ depending only on $s$, we can construct an $m$-generated convex sets $A_q^m$ such that
\[
A_q^m \subset A_q \subset A_q^{m,1/n}
\] 
and $\|v\|_0\leq s$ for all $v\in \mathcal{V}(A_q^m)$. The set $A^0=\cap_{q=1}^Q A_q^{m,1/n}$ trivially satisfies condition (C) with $a=0$ and $d$ depending only on $s$. In addition, it follows from the same arguments as those used in Step 1 that the set $A^0$ satisfies conditions (M.1$'$), (M.2$'$), (E.1$'$) (if (E.1) is satisfied), and (E.2$'$) (if (E.2) is satisfied). Therefore, by applying Proposition \ref{cor: sparsely convex bodies}, we conclude that $|\Pr(S_n^X\in A^0)-\Pr(S_n^Y\in A^0)|$
is bounded from above by the quantities on the right-hand sides of (\ref{eq: best rate}) and (\ref{eq: best rate 2}) depending on whether (E.1) or (E.2) is satisfied. Also, observe that $A\subset A^0$ and that $\cap_{q=1}^Q A_q^{m,-\varepsilon}$ is empty because $\cap_{q=1}^Q A_q^{m}\subset A$ and $A$ contains no ball with radius $\varepsilon$. This implies that $\Pr(S_n^Y\in A^0)\leq C (\log^{1/2} (pn))/n$ by Lemma \ref{lem: anti-concentration} and condition (M.1$''$). Since $A\subset A^0$, both $\Pr(S_n^X\in A)$ and $\Pr(S_n^Y\in A)$ are bounded from above by the quantities on the right-hand sides of (\ref{eq: sparsely convex sets with mgf}) and (\ref{eq: sparsely convex sets with finite moments}) depending on whether (E.1) or (E.2) is satisfied, and so is their difference. This completes the proof in this case.
\qed

%\begin{remark}
%[On condition (M.1$'$) for sparsely convex sets]
%As the proof of Lemma \ref{lem: verifying condition C.j} reveals, a simple sufficient condition for any $(s,Q)$-sparsely convex set $A$ in $\R^p$  to obey condition (M.1$'$) is that $n^{-1}\sum_{i=1}^n\Ep[(v'X_i)^2]\geq b$ for any $v\in \mathbb{S}^{p-1}$ with non-zero components corresponding to main components of some $A_q$ in the sparse representation $A=\cap_{q=1}^Q A_q$.
%\end{remark}

\medskip

Here we prove Lemma \ref{lem: verifying condition C.j} used in the proof of Proposition \ref{cor: sparsely convex sets}.

\begin{proof}[Proof of Lemma \ref{lem: verifying condition C.j}]
For convex sets $P_1$ and $P_2$ containing the origin and such that $P_1\subset P_2$, define
\[
d_{BM}(P_1,P_2):=\inf\{ \epsilon > 0: P_2\subset (1+\epsilon)P_1\}.
\]
It is immediate to verify that the function $d_{BM}$ has the following useful property: for any convex sets $P_1$, $P_2$, $P_3$, and $P_4$ containing the origin and such that $P_1\subset P_2$ and $P_3\subset P_4$,
\begin{equation}\label{eq: useful property of BM}
d_{BM}(P_1\cap P_3,P_2\cap P_4)\leq d_{BM}(P_1, P_2)\vee d_{BM}(P_3, P_4).
\end{equation}

Let $A=\cap_{q=1}^Q A_q$ be a sparse representation of $A$ as appeared in the statement of the lemma. Fix any $A_q$. By assumption, the indicator function $w\mapsto I(w\in A_q)$ depends only on $s_q\leq s$ elements of its argument $w=(w_1,\dots,w_p)$. %Let $E_q$ denote the subspace of $\R^p$ corresponding to these $s_q$ components. 
Since $A$ contains the origin, $A_q$ contains the origin as well. Therefore, applying Corollary 1.5 in \cite{B13} as if $A_q$ were a set in $\R^{s_q}$ shows that one can construct a polytope $P_q\subset \R^p$ with at most $(\gamma((\mu+1)/\epsilon)^{1/2}\log(1/\epsilon))^{s_q}$ vertices such that
\[
P_q\subset A_q\subset (1+\epsilon)P_q
\]
and such that for all $v\in\mathcal{V}(P_q)$, non-zero elements of $v$ correspond to some of the main components of $A_q$.
Since we need at most $s_q$ vertices to form a facet of the polytope $P_q$, the polytope $P_q$ has
\begin{equation}\label{eq: inequality for number of faces}
m_q\leq \left(\gamma\sqrt{\frac{\mu+1}{\epsilon}}\log\frac{1}{\epsilon}\right)^{s_q^2}\leq \left(\gamma\sqrt{\frac{\mu+1}{\epsilon}}\log\frac{1}{\epsilon}\right)^{s^2}
\end{equation}
facets. Now observe that $P_q$ is an $m_q$-generated convex set. Thus, we have constructed an $m_q$-generated convex set $P_q$ such that
$P_q\subset A_q\subset (1+\epsilon)P_q$ and all vectors in $\mathcal{V}(P_q)$ having at most $s$ non-zero elements. Hence $d_{BM}(P_q,A_q)\leq\epsilon$, 
which, together with (\ref{eq: useful property of BM}), implies that
\[
d_{BM}(\cap_{q=1}^Q P_q,\cap_{q=1}^Q A_q)\leq \epsilon.
\]
Therefore, defining $A^m=\cap_{q=1}^Q P_q$, we obtain from $A=\cap_{q=1}^Q A_q$ that
\[
A^m\subset A\subset (1+\epsilon) A^m\subset A^{m,R\epsilon},
\]
where the last assertion follows from the assumption that $\sup_{w\in A}\|w\|\leq R$. Since $A^m$ is an $m$-generated convex set with $m\leq \sum_{q=1}^Q m_q$, the first claim of the lemma now follows from (\ref{eq: inequality for number of faces}). The second claim (\ref{eq: useful sparsity condition}) holds by construction of $A^m$, and the final claim is trivial. 
\end{proof}

\section{Proofs for Section \ref{sec: bootstrap CLT}}

\subsection{Maximal inequalities}

Here we collect some useful maximal inequalities that will be used in the proofs for Section \ref{sec: bootstrap CLT}.

\begin{lemma}
\label{lem: maximal ineq}
Let $X_{1},\dots,X_{n}$ be independent centered random vectors in $\R^{p}$ with $p \geq 2$. Define $Z := \max_{1 \leq j \leq p} | \sum_{i=1}^{n} X_{ij}|, \ M := \max_{1 \leq i \leq n} \max_{1 \leq j \leq p} | X_{ij} |$ and $\sigma^{2} := \max_{1 \leq j \leq p} \sum_{i=1}^{n} \Ep [ X_{ij}^{2} ]$. Then
\begin{equation*}
\Ep [Z]  \leq K (\sigma \sqrt{\log p} + \sqrt{\Ep [ M^{2} ]} \log p).
\end{equation*}
where $K$ is a universal constant.
\end{lemma}

\begin{proof}
See Lemma 8 in \cite{CCK12c}.
\end{proof}

\begin{lemma}
\label{lem: fuk-nagaev}
Assume the setting of Lemma \ref{lem: maximal ineq}. 
(i) For every $\eta > 0, \beta \in (0,1]$ and $t>0$,
\[
\Pr \{  Z \geq (1+\eta) \Ep [ Z ] + t \} \leq \exp \{ -t^{2}/(3\sigma^{2}) \} +3\exp \{ -( t/(K\| M \|_{\psi_{\beta}}))^{\beta} \},
\]
where $K = K(\eta,\beta)$ is a constant depending only on $\eta, \beta$. 

(ii) For every $\eta >0, s \geq 1$ and $t>0$,
\[
\Pr \{ Z \geq (1+\eta) \Ep[Z] + t \} \leq \exp\{ -t^2/(3\sigma^2)\}+K' \Ep [ M^{s}]/t^s, 
\]
where $K' =K'(\eta,s)$ is a constant depending only on $\eta,s$. 
\end{lemma}
\begin{proof}
See Theorem 4 in \cite{A08a} for case (i) and Theorem 2 in \cite{A08} for case (ii). See also \cite{EL08}.
\end{proof}

\begin{lemma}
\label{lem: maximal ineq nonnegative}
Let $X_1,\dots,X_n$ be independent random vectors in $\R^p$ with $p\geq 2$ such that $X_{ij} \geq 0$ for all $i=1,\dots,n$ and $j=1,\dots,p$. Define $Z := \max_{1 \leq j \leq p} \sum_{i=1}^{n} X_{ij}$ and $M := \max_{1 \leq i \leq n} \max_{1 \leq j \leq p}  X_{ij} $. Then
\[
\Ep [Z] \leq K \left( \max_{1\leq j\leq p}\Ep [ {\textstyle \sum}_{i=1}^n X_{ij} ]+\Ep [M] \log p \right),
\]
where $K$ is a universal constant. 
\end{lemma}
\begin{proof}
See Lemma 9 in \cite{CCK12c}.
\end{proof}

\begin{lemma}
\label{lem: deviation ineq nonnegative}
Assume the setting of Lemma \ref{lem: maximal ineq nonnegative}.
(i) For every $\eta > 0, \beta \in (0,1]$ and $t > 0$, 
\[
\Pr \{ Z \geq (1+\eta) \Ep[Z] + t \} \leq 3 \exp \{ -( t/(K\| M \|_{\psi_{\beta}}))^{\beta} \},
\]
where $K = K(\eta,\beta)$ is a constant depending only on $\eta, \beta$. 
(ii) For every $\eta > 0, s \geq 1$ and $t > 0$, 
\[
\Pr \{ Z  \geq (1+\eta) \Ep[Z] + t \} \leq K'  \Ep[M^{s}]/t^{s},
\]
where $K' = K'(\eta,s)$ is a constant depending only on $\eta,s$. 
\end{lemma}

The proof of Lemma \ref{lem: deviation ineq nonnegative} relies on the following lemma, which follows from Theorem 10 in \cite{M00}. 
\begin{lemma}
\label{lem: talagrand nonnegative}
Assume the setting of Lemma \ref{lem: maximal ineq nonnegative}. Suppose that there exists a constant $B$ such that $M \leq B$. Then for every $\eta, t > 0$,
\[
\Pr \left \{ Z \geq (1+\eta) \Ep[Z]  + B \left (\frac{2}{3} + \frac{1}{\eta} \right ) t \right \} \leq e^{-t}. 
\]
\end{lemma}
\begin{proof}[Proof of Lemma \ref{lem: talagrand nonnegative}]
By homogeneity, we may assume that $B=1$. Then by Theorem 10 in \cite{M00}, for every $\lambda > 0$,
\[
\log \Ep [\exp (\lambda (Z-\Ep[Z])) ] \leq \varphi (\lambda) \Ep[Z], 
\]
where $\varphi (\lambda) = e^{\lambda} - \lambda -1$. Hence by Markov's inequality, with $a = \Ep[Z]$, 
\[
\Pr \{ Z-\Ep[Z] \geq t \} \leq e^{-\lambda t + a \varphi (\lambda) }.
\]
The right-hand side is minimized at $\lambda = \log ( 1+t/a )$, at which $-\lambda t + a \varphi (\lambda)=-a q(t/a)$ where $q(t) = (1+t)\log (1+t) - t$. It is routine to verify that $q(t) \geq t^{2}/(2(1+t/3))$, so that 
\[
\Pr \{ Z-\Ep[Z] \geq t \} \leq e^{-\frac{t^{2}}{2(a+t/3)}}. 
\]
Solving $t^{2}/(2(a+t/3)) = s$
gives $t = s/3 + \sqrt{s^{2}/9+2as} \leq 2s/3 + \sqrt{2as}$. Therefore, we have 
\[
\Pr \{ Z \geq \Ep[Z] + \sqrt{2as} + 2 s/ 3 \} \leq e^{-s}. 
\]
The conclusion follows from the inequality $\sqrt{2as} \leq \eta a + \eta^{-1}s$. 
\end{proof}

\begin{proof}[Proof of Lemma \ref{lem: deviation ineq nonnegative}]
The proof is a modification of that of Theorem 4 in \cite{A08a} (or Theorem 2 in \cite{A08}). 
We begin with noting that we may assume that $(1+\eta) 8 \Ep[M] \leq t/4$, since otherwise we can make the lemma trivial by setting $K$ or $K'$ large enough. 
Take 
\[
\rho = 8 \Ep[ M ], \ Y_{ij}=
\begin{cases}
X_{ij}, & \text{if} \ \max_{1 \leq j \leq p} X_{ij} \leq \rho, \\
0, & \text{otherwise}
\end{cases}
\]
Define 
\[
W_{1}= \max_{1 \leq j \leq p} \sum_{i=1}^{n} Y_{ij}, \ 
W_{2} = \max_{1 \leq j \leq p} \sum_{i=1}^{n} (X_{ij}-Y_{ij}). 
\]
Then 
\begin{align*}
&\Pr \{ Z \geq (1+\eta) \Ep[Z] + t \} \leq \Pr \{ W_{1} \geq (1+\eta) \Ep [Z] + 3 t/4 \} + \Pr (W_{2} \geq t/4) \\
&\quad \leq \Pr \{ W_{1} \geq (1+\eta) \Ep[W_{1}] - (1+\eta) \Ep[W_{2}] + 3 t/4 \} + \Pr (W_{2} \geq t/4).
\end{align*}
Observe that 
\[
\Pr \left \{ \max_{1 \leq m \leq n} \max_{1 \leq j \leq p} \sum_{i=1}^{m}(X_{ij}-Y_{ij}) > 0 \right \} \leq \Pr (M > \rho) \leq 1/8,
\]
so that by the Hoffmann-J\o rgensen inequality \citep[see][Proposition 6.8]{LT91}, we have 
\[
\Ep[ W_{2} ] \leq 8 \Ep[M] \leq t/(4(1+\eta)). 
\]
Hence 
\[
\Pr \{ Z \geq (1+\eta) \Ep[Z] + t \} \leq \Pr \{ W_{1} \geq (1+\eta) \Ep[W_{1}] + t/2 \} + \Pr (W_{2} \geq t/4).
\]
By Lemma \ref{lem: talagrand nonnegative}, the first term on the right-hand side is bounded by $e^{-c t/\rho}$ where $c$ depends only on $\eta$. We bound the second term separately in cases (i) and (ii). Below $C_{1},C_{2},\dots$ are constants that depend only on $\eta, \beta, s$. 

Case (i).  By Theorem 6.21 in \cite{LT91} (note that a version of their theorem applies to nonnegative random vectors) and the fact that $\Ep[ W_{2}] \leq 8\Ep[M]$, 
\[
\| W_{2} \|_{\psi_{\beta}} \leq C_{1}(\Ep[W_{2}] + \| M \|_{\psi_{\beta}}) \leq C_{2} \| M \|_{\psi_{\beta}}, 
\]
which implies that $\Pr (W_{2} \geq t/4) \leq 2\exp \{ - (t/(C_{3}\| M \|_{\psi_{\beta}}))^{\beta} \}$. Since $\rho \leq C_{4} \| M \|_{\psi_{\beta}}$, we conclude that 
\[
e^{-ct/\rho} + \Pr (W_{2} \geq t/4) \leq 3\exp \{ - (t/(C_{5}\| M \|_{\psi_{\beta}}))^{\beta} \}.
\]

Case (ii). By Theorem 6.20 in \cite{LT91} (note that a version of their theorem applies to nonnegative random vectors) and the fact that $\Ep[ W_{2}] \leq 8\Ep[M]$, 
\[
(\Ep[W_{2}^{s}])^{1/s} \leq C_{6} (\Ep[W_{2}] + (\Ep[M^{s}])^{1/s}) \leq C_{7} (\Ep[M^{s}])^{1/s}. 
\]
The conclusion follows from Markov's inequality together with the simple fact that $e^{-t}/t^{-s} \to 0$ as $t \to \infty$. 
\end{proof}

\subsection{Proofs for Section \ref{sec: bootstrap CLT}}

\subsection*{Proof of Theorem \ref{thm: multiplier bootstrap CLT}}
In this proof, $C$ is a positive constant that depends only on $a$, $b$, and $d$ but its value may change at each appearance. Fix any $A\in\mA\subset \mA^{\si}(a,d)$. Let $A^m=A^m(A)$ be an approximating $m$-generated convex set as in (C). By assumption, $A^m\subset A\subset A^{m,\epsilon}$. Let
\begin{align*}
\overline{\rho}:= \max \Big \{ & |\Pr(S_n^{eX}\in A^{m} \mid X_1^n)-\Pr(S_n^Y\in A^{m})|, \\
&\quad |\Pr(S_n^{eX}\in A^{m,\epsilon} \mid X_1^n)-\Pr(S_n^Y\in A^{m,\epsilon})| \Big \}.
\end{align*}
As in the proof of Proposition \ref{cor: sparsely convex bodies}, we have 
\begin{align*}
&|\Pr(S_n^{eX}\in A\mid X_1^n)-\Pr(S_n^Y\in A)|  \\
&\quad \leq C\epsilon\log^{1/2} (p n)+\overline{\rho}
\leq Cn^{-1}\log^{1/2} (p n) + \overline{\rho}, 
\end{align*}
so that the problem reduces to proving 
%\[
%\rho_n^{MB}(A^m)\leq C\overline{\Delta}_n^{1/3}\log^{2/3} (p n).
%\]
%Hence, it suffices to prove 
that under (M.1), the inequality
\begin{equation}\label{eq: to prove 4-1}
\rho_n^{MB}(\mA^{\re})\leq C\overline{\Delta}_n^{1/3}\log^{2/3} p
\end{equation}
holds on the event $\Delta_{n,r}\leq\overline{\Delta}_n$, where $\Delta_{n,r}:=\max_{1\leq j,k\leq p}|\hat{\Sigma}_{jk}-\Sigma_{jk}|$ 
with $\hat{\Sigma}_{jk}$ and $\Sigma_{jk}$ denoting the $(j,k)$-th elements $\hat{\Sigma}$ and $\Sigma$, respectively.

To this end, we first show that
\begin{equation}\label{eq: varrho bootstrap}
\varrho_n^{MB}:=\sup_{y\in\R^p}| \Pr(S_n^{e X}\leq y \mid X_{1}^{n})-\Pr(S_n^Y\leq y) |\leq C\Delta_{n,r}^{1/3}\log^{2/3} p.
\end{equation}
To show (\ref{eq: varrho bootstrap}), fix any $y=(y_1,\dots,y_p)'\in \R^p$. As in the proof of Lemma \ref{lem: induction lemma}, for $\beta>0$, define
\[
F_\beta(w):=\beta^{-1}\log \left({\textstyle \sum}_{j=1}^p\exp(\beta(w_j-y_j))\right), \ w \in \R^{p}.
\]
Note that conditional on $X_{1}^{n}$, $S_n^{e X}$ is a centered Gaussian random vector with covariance matrix $\hat{\Sigma}$. 
%Hence, using the same argument as that in the proof of Theorem 1 in , 
Then a small modification of the proof of Theorem 1 in \cite{CCK12c} implies that for every  $g\in C^2(\R)$ with $\|g'\|_{\infty}\vee\|g''\|_{\infty}<\infty$, we have
\[
|\Ep[g(F_\beta(S_n^{e X})) \mid X_{1}^{n}]-\Ep[g(F_\beta(S_n^Y))]|\leq (\|g''\|_\infty/2+\beta\|g'\|_\infty)\Delta_{n,r}.
\]
%(Formally, the proof of Theorem 1 in \cite{CCK12c} imposes $y=0$ but it is easy to verify that their proof also applies for all $s\in\R^p$.)
Hence, as in Step 2 of the proof of Lemma \ref{lem: induction lemma}, we obtain with $\phi=\beta/\log p$ that
\begin{align*}
&|\Pr(S_n^{e X}\leq y-\phi^{-1} \mid X_{1}^{n})-\Pr(S_n^Y\leq y-\phi^{-1})| \\
&\quad \leq C\{ \phi^{-1}\log^{1/2}p+(\phi^2+\beta\phi)\Delta_{n,r}\}.
\end{align*}
Substituting $\beta=\phi\log p$, optimizing the resulting expression with respect to $\phi$, and noting that $y \in \R^{p}$ is arbitrary lead to (\ref{eq: varrho bootstrap}). 
Finally (\ref{eq: to prove 4-1}) follows from the fact that the inequality  $\varrho_n^{MB}\leq C\overline{\Delta}_{n}^{1/3}\log^{2/3} p$ holds on the event $\Delta_{n,r}\leq \overline{\Delta}_n$, and applying the same argument as that used in the proof of Corollary \ref{cor: induction lemma}.  
\qed

\subsection*{Proof of Proposition \ref{cor: MB corollary}}
In this proof, $c$ and $C$ are positive constants that depend only on $a,b,d$, and $s$ under (E.1), and on $a,b,d,s$, and $q$ under (E.2); their values may vary from place to place. For brevity of notation, we implicitly assume here that $i$ is varying over $\{1,\dots,n\}$, and $j$ and $k$ are varying over $\{1,\dots,p\}$. 
Finally, without loss of generality, we will assume that
\begin{equation}\label{eq: trivial bound 0}
B_n^2\log^5(p n)\log^2(1/\alpha) \leq n
\end{equation}
since otherwise the assertions are trivial. 

We shall apply Theorem \ref{thm: multiplier bootstrap CLT} to prove the proposition. Observe that since $n^{-1} \log^{1/2}(pn) \leq C D_{n}^{(1)}(\alpha)$, 
it suffices to construct an appropriate $\overline{\Delta}_n$ such that $\Pr(\Delta_n(\mA)>\overline{\Delta}_n)\leq \alpha$ and to bound  $\overline{\Delta}_n^{1/3}\log^{2/3}(p n)$. 

We begin with noting that since (S) holds for all $A\in\mA$, $\Delta_n(\mA) \leq C\Delta_{n,r}$ where $\Delta_{n,r}=\max_{1\leq j,k\leq p}|\hat{\Sigma}_{jk}-\Sigma_{jk}|$.
As $\hat{\Sigma}-\Sigma=n^{-1}\sum_{i=1}^n(X_i X_i'-\Ep[X_i X_i'])-\bar{X}\bar{X}'$, we have $\Delta_{n,r}\leq \Delta_{n,r}^{(1)}+\{ \Delta_{n,r}^{(2)} \}^2$,
where
\[
\Delta_{n,r}^{(1)}:=\max_{1\leq j,k\leq p}\left|n^{-1}\sum_{i=1}^n(X_{ij} X_{ik}-\Ep[X_{ij} X_{ik}])\right|, \ \Delta_{n,r}^{(2)}:= \max_{1\leq j\leq p}|\bar{X}_{j}|.
\]
The desired assertions then follow from the bounds on $\Delta_{n,r}^{(1)}$ and $\Delta_{n,r}^{(2)}$ derived separately for (E.1) and (E.2) cases below.

\medskip

{\bf Case (E.1)}. 
Observe that by H\"{o}lder's inequality and (M.2),
\[
\sigma_{n}^2:=\max_{j,k}\sum_{i=1}^n\Ep\left[(X_{ij}X_{ik}-\Ep[X_{ij}X_{ik}])^2\right]\leq \max_{j,k}\sum_{i=1}^n \Ep[|X_{ij}X_{ik}|^2] \leq nB_{n}^{2}.
% &\leq \max_{j,k}\left(\sum_{i=1}^n\Ep[|X_{i j}|^4]\sum_{i=1}^n\Ep[|X_{ik}|^4]\right)^{1/2}\leq nB_n^2. 
\]
In addition, by (E.1), 
\[
\|\max_{i,j,k}|X_{ij}X_{ik}|\|_{\psi_{1/2}}=\|\max_{i,j}|X_{ij}|^2\|_{\psi_{1/2}}=\|\max_{i,j}|X_{ij}|\|_{\psi_1}^2 \leq C B_{n}^{2}\log^{2}(pn),
\]
so that for $M_{n}:=\max_{i,j,k}|X_{ij}X_{ik}-\Ep[X_{ij}X_{ik}]|$, we have 
\begin{align*}
&\|M_n\|_{\psi_{1/2}} \leq C \{ \|\max_{i,j,k}|X_{ij}X_{ik}|\|_{\psi_{1/2}}+\max_{i,j,k}\Ep[|X_{ij}X_{ik}|] \} \\
&\quad \leq C \{ B_{n}^{2}\log^{2}(pn)+B_n^2 \} \leq C B_{n}^{2}\log^{2}(pn),
\end{align*}
which also implies that $(\Ep[M_n^2])^{1/2} \leq CB_{n}^{2}\log^{2} (pn)$.
Hence by Lemma \ref{lem: maximal ineq}, we have
\begin{align*}
&\Ep[\Delta_{n,r}^{(1)}] \leq Cn^{-1}\{ \sqrt{ \sigma_{n}^2\log p} +\sqrt{\Ep[M_{n}^2]}\log p \} \\
&\quad \leq C \{ (n^{-1}B_n^2\log p)^{1/2}+ n^{-1} B_n^2\log^{3}(p n) \} \leq C \{ n^{-1}B_n^2\log (pn) \}^{1/2},
\end{align*}
where the last inequality follows from (\ref{eq: trivial bound 0}).
Applying Lemma \ref{lem: fuk-nagaev} (i) with $\beta=1/2$ and $\eta = 1$, we conclude that for every $t > 0$,
\begin{align*}
&\Pr\left( \Delta_{n,r}^{(1)} > C \{ n^{-1}B_n^2\log (pn) \}^{1/2}+t\right) \\
&\quad \leq \exp \{-nt^{2}/(3B_{n}^{2})\}+3 \exp \{ -c\sqrt{nt}/(B_n\log(pn)) \}.
\end{align*}
Choosing $t=C\{n^{-1}B_n^2\log(p n)\log^{2}(1/\alpha)\}^{1/2}$ for sufficiently large $C>0$, the right-hand side of this inequality is bounded by 
\[
\alpha/4+3\exp \{ - cC^{1/2}n^{1/4}\log^{1/2}(1/\alpha)/( B_n^{1/2}\log^{3/4}(p n)) \} \leq \alpha/2,
\]
where the last inequality follows from (\ref{eq: trivial bound 0}).
Therefore
\[
\Pr(  \{ \Delta_{n,r}^{(1)}\log^{2} (pn)\}^{1/3}> C D_{n}^{(1)}(\alpha) ) \leq \alpha/2.
\]
It is routine to verify that the same inequality holds with $\Delta_{n,r}^{(1)}$ replaced by $\{ \Delta_{n,r}^{(2)} \}^2$. 
This leads to the conclusion of the proposition under (E.1). 
% Therefore, the asserted claim for the (E.1) case follows from Theorem \ref{thm: multiplier bootstrap CLT} applied with $\overline{\Delta}_n=\overline{\Delta}_{n,1}+\overline{\Delta}_{n,2}^2$ where $\overline{\Delta}_{n,1}$ and $\overline{\Delta}_{n,2}$ are such that $(\overline{\Delta}_{n,1}(\log p)^2)^{1/3}$ and $(\overline{\Delta}_{n,2}^2(\log p)^2)^{1/3}$ are both equal to the right-hand side expression under the probability sign in (\ref{eq: final prob bound E1 3-1}).

\medskip

{\bf Case (E.2)}.
Define $\sigma_n^2$ and $M_n$ by the same expressions as those in the previous case; then $\sigma_{n}^{2} \leq nB_{n}^{2}$. For $M_n$, we have 
\begin{align*}
\Ep[M_n^{q/2}] &\leq C\{ \Ep[\max_{i,j,k}|X_{ij}X_{ik}|^{q/2}]+\max_{i,j,k}\left(\Ep[|X_{ij}X_{ik}|]\right)^{q/2} \} \\
&\leq C\{ \Ep[\max_{i,j,k}|X_{ij}X_{ik}|^{q/2}] \} =C\Ep[\max_{i,j} |X_{ij}|^q] \leq C n B_n^q,
\end{align*}
which also implies that $(\Ep[M_n^2])^{1/2} \leq C n^{2/q}B_n^2$.
Hence by Lemma \ref{lem: maximal ineq}, we have
\begin{align*}
&\Ep[\Delta_{n,r}^{(1)}] \leq C n^{-1} \{ \sqrt{\sigma_{n}^2\log p}+\sqrt{\Ep[M_{n}^2]}\log p \} \\
&\quad \leq C \{( n^{-1} B_n^2\log p )^{1/2}+n^{-1+2/q} B_n^2\log p\}.
\end{align*}
Applying Lemma \ref{lem: fuk-nagaev}(ii) with $s=q/2$ and $\eta=1$, we have for every $t>0$,
\begin{align*}
&\Pr \left \{ \Delta_{n,r}^{(1)} >C \{( n^{-1} B_n^2\log p )^{1/2}+n^{-1+2/q} B_n^2\log p\} +t \right \}\\
&\quad \leq \exp \{-nt^{2}/(3B_{n}^{2})\} + c t^{-q/2} n^{1-q/2}B_n^q.
\end{align*}
Choosing 
\[
t=C\{ \{ n^{-1} B_n^2(\log(p n)) \log^{2}(1/\alpha) \}^{1/2}+n^{-1+2/q} \alpha^{-2/q} B_n^2 \}
\]
for  sufficiently large $C>0$, we conclude that
\[
\Pr \Big(  \{ \Delta_{n,r}^{(1)}\log^{2} (pn)\}^{1/3}> C \{ D_{n}^{(1)}(\alpha) + D_{n,q}^{(2)}(\alpha) \} \Big) \leq \alpha/2.
\]
% $$
% \Pr\left(\Delta_{n,1}>C\left(\frac{B_n^2\log (p n)(\log(1/\alpha))^2}{n}\right)^{1/2}+\frac{C B_n^2\log p}{\alpha^{2/q}n^{1-2/q}}\right)\leq \frac{\alpha}{4}+\frac{\alpha}{4}=\frac{\alpha}{2}
% $$
% by (\ref{eq: trivial bound 1}).
% Hence, using an elementary inequality $|x+y|^{1/3}\leq |x|^{1/3}+|y|^{1/3}$ for all $x,y\in\R$, we obtain that the probability that $(\Delta_{n,1}(\log p)^2)^{1/3}$ is larger than
% $$
% C\left[\left(\frac{B_n^2(\log (p n))^5(\log(1/\alpha))^2}{n}\right)^{1/6}+\left(\frac{B_n^2(\log p)^3}{\alpha^{2/q}n^{1-2/q}}\right)^{1/3}\right]
% $$
% is bounded from above by $\alpha/2$.
It is routine to verify that the same inequality holds with $\Delta_{n,r}^{(1)}$ replaced by $\{ \Delta_{n,r}^{(2)} \}^2$.  This leads to the conclusion of the proposition under (E.2). \qed

\subsection*{Proof of Corollary \ref{cor: MB corollary log-concave distributions}}
Here $C$ is understood to be a positive constant that depends only on $a,d,k_1$ and $k_2$; the value of $C$ may change from place to place.  
To prove this corollary, we apply Theorem \ref{thm: multiplier bootstrap CLT}, to which end we have to verify condition (M.1$'$) for all $A\in\mA$ and derive a suitable bound on $\Delta_{n}(\mA)$. 
Condition (M.1$'$) for all $A\in\mA$ follows from the fact that the minimum eigenvalue of $\Ep[X_{i}X_{i}']$ is bounded from below by $k_{1}$. 
By log-concavity of the distributions of $X_{i}$, we have $\| v'X_{i} \|_{\psi_{1}} \leq C (\Ep[(v'X_{i})^{2}])^{1/2} \leq C$ for all $v \in \R^{p}$ with $\| v \| = 1$ (see the proof of Corollary \ref{cor: log-concave distributions}). For all $i=1,\dots,n$, let $\check{X}_{i}$ be a random vector whose elements are given by $v'X_{i}, v \in \cup_{A \in \mA}\mathcal{V}(A^{m}(A))$; the dimension of $\check{X}_{i}$, denoted by $\check{p}$, is at most $(pn)^{d}$, and$\| \check{X}_{ij} \|_{\psi_{1}} \leq C$ for all $j=1,\dots,\check{p}$. Then $\Delta_{n}(\mA)$ coincides with $\Delta_{n,r}$ with $X_{i}$ replaced by $\check{X}_{i}$, that is,
\[
\Delta_{n}(\mA) = \max_{1 \leq j,k \leq \check{p}} \left | n^{-1}\sum_{i=1}^{n} (\check{X}_{ij} \check{X}_{ik} - \Ep[\check{X}_{ij}\check{X}_{ik}]) - \En[\check{X}_{ij}] \En[\check{X}_{ik}] \right |.
\]
Noting that $\log \check{p} \leq d \log (pn)$, by the same argument as that used in the proof of Proposition \ref{cor: MB corollary} case (E.1), we can find a constant $\overline{\Delta}_{n}$ such that $\Pr (\Delta_{n}(\mA) > \overline{\Delta}_{n}) \leq \alpha$ and 
\[
\{ \overline{\Delta}_{n}\log^{2} (pn) \}^{1/3} \leq C \{ n^{-1}(\log^{5} (pn))\log^{2} (1/\alpha) \}^{1/6}.
\]
Here without loss of generality we assume that $(\log^{5} (pn)) \log^{2} (1/\alpha) \leq n$. The desired assertion then follows. 
% By the same arguments as those used in the proof of Corollary \ref{cor: log-concave distributions}, it follows that any simple convex set $A\in\mA^{\si}$ obeys conditions (M.1$'$), (M.2$'$), and (E.1$'$) with $B_n$ replaced by a constant depending only on $k_2$ and with $b$ replaced by a constant depending only on $k_1$. Therefore, the asserted claim follows by applying Proposition \ref{cor: MB corollary}.
\qed

\subsection*{Proof of Corollary \ref{cor: MB for rectangles explicit}}
Any hyperrectangle $A\in\mA^{\re}$ satisfies conditions (C) and (S) with $a=0$, $d=1$, and $s=1$. In addition, it follows from (M.1) that any hyperrectangle $A\in\mA^{\re}$ satisfies (M.1$'$). Therefore, the asserted claims follow from Proposition \ref{cor: MB corollary}.
\qed

\subsection*{Proof of Proposition \ref{cor: MB corollary sparsely convex sets}}
In this proof, let $C$ be a positive constant depending only on $b$ and $s$ under (E.1), and on $b$, $q$, and $s$ under (E.2); the value of $C$ may change from place to place. 
Moreover, without loss of generality, we will assume that
% \begin{equation}
% \label{eq: trivial bound 1}
\[
B_n^2(\log^5(p n))\log^2(1/\alpha) \leq n
\]
% \end{equation}
since otherwise the assertions are trivial. 

Let $\Delta_{n,r} := \max_{1 \leq j,k \leq p} |\hat{\Sigma}_{jk} - \Sigma_{jk}|$, and 
\[
\overline{\Delta}_{n} = 
\begin{cases}
\left ( \frac{B_{n}^{2} (\log (pn)) \log^{2} (1/\alpha)}{n} \right )^{1/2} & \text{if (E.1) is satisfied} \\
\left ( \frac{B_{n}^{2} (\log (pn)) \log^{2} (1/\alpha)}{n} \right )^{1/2} + \frac{B_{n}^{2}\log p}{\alpha^{2/q} n^{1-q/2}} & \text{if (E.2) is satisfied}.
\end{cases}
\]
Then by the proof of Proposition \ref{cor: MB corollary}, in either case where (E.1) or (E.2) is satisfied, there exists a positive constant $C_{1}$ depending only on $b,s,q$ ($C_{1}$ depends on $q$ only in the case where  (E.2) is satisfied) such that  
\[
\Pr (\Delta_{n,r} > C_{1} \overline{\Delta}_{n}) \leq \alpha/2.
\]
We may further assume that $C_{1} \overline{\Delta}_{n} \leq b/2$, since otherwise the assertions are trivial.

As in the proof of Proposition \ref{cor: sparsely convex sets}, let $R=p n^{5/2}$ and $V^{R} = \{ w \in \mathbb{R}^p: \max_{1 \leq j \leq p} |w_{j}|> R \}$. Fix any $A \in \mA^{\spa}(s)$. Then $A = \check{A} \cup (A \cap V^{R})$ for some $s$-sparsely convex set $\check{A}$ with $\sup_{w\in\check{A}}\max_{1 \leq j \leq p} |w_{j}| \leq R$. As in Proposition \ref{cor: sparsely convex sets}, $\Pr (S_{n}^{Y} \in V^{R}) \leq C/n^{1/2}$. Moreover, conditional on $X_{1}^{n}$, $S_{nj}^{eX}$ is Gaussian with mean zero and variance $\En[(X_{ij}-\bar{X}_{j})^{2}] = \hat{\Sigma}_{jj}$, so that 
\begin{align*}
&\Pr (S_{n}^{eX} \in V^{R} \mid X_{1}^{n}) = \Pr (\max_{1 \leq j \leq p}|S_{nj}^{eX}| > R \mid X_{1}^{n}) \\
&\qquad \leq \frac{\Ep[\max_{1 \leq j \leq p}|S_{nj}^{eX}| \mid X_{1}^{n}]}{R} \leq \frac{C (\log p)^{1/2} \max_{1 \leq j \leq p} \hat{\Sigma}_{jj}^{1/2}}{R},
\end{align*}
which is bounded by $C/n^{1/2}$ on the event $\Delta_{n,r} \leq C_{1} \overline{\Delta}_{n}$. Hence on the event $\Delta_{n,r} \leq C_{1} \overline{\Delta}_{n}$, 
\begin{align*}
&|\Pr (S_{n}^{eX} \in A \mid X_{1}^{n}) - \Pr (S_{n}^{Y} \in A) | \\
&\qquad\leq | \Pr(S_{n}^{eX} \in \check{A} \mid X_{1}^{n}) - \Pr(S_{n}^{Y} \in \check{A})| + C/n^{1/2},
\end{align*}
so that it suffices to consider the case where the sets $A \in \mA^{\spa}(s)$ are such that $\sup_{w\in A}\max_{1 \leq j \leq p}|w_{j}| \leq R$. 

Further, let $\varepsilon=n^{-1}$, and define the subclasses $\mA^{\spa}_1(s)$ and $\mA^{\spa}_2(s)$ of $\mA^{\spa}(s)$ as in the proof of Proposition \ref{cor: sparsely convex sets}. For all $A\in\mA^{\spa}_1(s)$, we can verify conditions (C), (S), and (M.1$'$) as in the proof of Proposition \ref{cor: sparsely convex sets} (where (S) is verified implicitly). Therefore, by Proposition \ref{cor: MB corollary} applied with $\alpha/2$ instead of $\alpha$, the bounds (\ref{eq: MB result E1 sparse}) and (\ref{eq: MB result E2 sparse}) with $\rho_n^{M B}(\mA^{\spa}(s))$ replaced by  $\rho_n^{M B}(\mA^{\spa}_1(s))$ hold with probability at least $1-\alpha/2$. Hence, it remains to bound $\rho_n^{M B}(\mA^{\spa}_2(s))$.

Fix any $A\in\mA^{\spa}_2(s)$ with a sparse representation $A=\cap_{q=1}^Q A_q$ for some $Q\leq p^s$. As in the proof of Proposition \ref{cor: sparsely convex sets}, we separately consider two cases. First, suppose that at least one of $A_{q}$ does not contain a ball of radius $\varepsilon$; then by condition (M.1$''$) and Lemma \ref{lem: Nazarov anti-concentration2}, $\Pr (S_{n}^{Y} \in A_{q}) \leq C \varepsilon$. Moreover, since $S_{n}^{eX}$ is Gaussian conditional on $X_{1}^{n}$, by condition (M.1$''$) and Lemma \ref{lem: Nazarov anti-concentration2}, we have, on the event $\Delta_{n,r} \leq C_{1} \overline{\Delta}_{n}$, $\Pr (S_{n}^{eX} \in A_{q} \mid X_{1}^{n}) \leq C\varepsilon$ since $C_1\overline{\Delta}_n\leq b/2$. Since $A \subset A_{q}$, we conclude that on the event $\Delta_{n,r} \leq C_{1} \overline{\Delta}_{n}$, $| \Pr (S_{n}^{eX} \in A \mid X_{1}^{n}) - \Pr(S_{n}^{Y} \in A)| \leq C \varepsilon = C/n$. 

Second, suppose that each $A_{q}$ contains a ball with radius $\varepsilon$. Then by applying Lemma \ref{lem: verifying condition C.j} to each $A_{q}$, for $m \leq (pn)^{d}$ with $d$ depending only on $s$, we can construct an $m$-generated convex set $A_{q}^{m}$ such that $A_{q}^{m} \subset A_{q} \subset A_{q}^{m,1/n}$ with $\| v \|_{0} \leq s$ for all $v \in \mathcal{V}(A_{q}^{m})$. Let $A_{0} = \cap_{q=1}^{Q} A_{q}^{m,1/n}$; then $A \subset A^{0}$ and $\cap_{q=1}^{Q} A_{q}^{m,-\varepsilon}$ is empty.  By the latter fact, together with condition (M.1$''$) and Lemma \ref{lem: anti-concentration}, we have $\Pr (S_{n}^{Y} \in A^{0}) \leq C(\log^{1/2}(pn))/n$. Moreover, since $S_{n}^{eX}$ is Gaussian conditional on $X_{1}^{n}$, by condition (M.1$''$) and Lemma \ref{lem: anti-concentration}, the inequality $\Pr (S_{n}^{eX} \in A^{0} \mid X_{1}^{n}) \leq C(\log^{1/2}(pn))/n$ holds on the event $\Delta_{n,r} \leq C_{1} \overline{\Delta}_{n}$ since $C_1\overline{\Delta}_n\leq b/2$.  Since $A \subset A^{0}$, we conclude that on the event $\Delta_{n,r} \leq C_{1} \overline{\Delta}_{n}$, $| \Pr (S_{n}^{eX} \in A \mid X_{1}^{n}) - \Pr(S_{n}^{Y} \in A)| \leq C(\log^{1/2} (pn))/n$. This completes the proof since $\Pr (\Delta_{n,r} > C_{1}\overline{\Delta}_{n}) \leq \alpha/2$. 
\qed

\subsection*{Proof of Theorem \ref{thm: empirical bootstrap CLT}}
By the triangle inequality, $\rho_n^{EB}(\mA^{\re})\leq \rho_n^{MB}(\mA^{\re})+\varrho_n^{EB}(\mA^{\re})$, 
where
\[
\varrho_n^{EB}(\mA^{\re}):=\sup_{A\in\mA^{\re}}|\Pr(S_n^{X^*}\in A \mid X_{1}^{n})-\Pr(S_n^{e X}\in A \mid X_{1}^{n})|.
\]
Also conditional on $X_{1}^{n}$, $X_1^*-\bar{X},\dots,X_n^*-\bar{X}$ are i.i.d. with mean zero and covariance matrix $\hat{\Sigma}$. 
In addition, conditional on $X_{1}^{n}$, $S_n^{eX} \stackrel{d}{=} \sum_{i=1}^n Y_i^*/\sqrt{n}$, 
where $Y_1^*,\dots,Y_n^*$ are i.i.d. centered Gaussian random vectors with the same covariance matrix $\hat{\Sigma}$. Hence the conclusion of the theorem follows from applying Theorem \ref{thm: main} conditional on $X_{1}^{n}$ (with $L_n$ and $M_n(\phi_n)$ in Theorem \ref{thm: main} substituted by $\hat{L}_n$ and $\hat{M}_n(\phi_n)$) to bound $\varrho_n^{EB}(A^{\re})$ on the event $\{ \En[(X_{ij}-\bar{X}_{j})^2]\geq b \ \text{for all} \ 1\leq j\leq p \} \cap \{ \hat{L}_n \leq \overline{L}_{n} \} \cap \{ \hat{M}_n(\phi_n) \leq \overline{M}_n \}$.
%Note that conditional on $(X_i)_{i=1}^n$, the sequence $X_1^*-\hat{\mu}_n,\dots,X_n^*-\hat{\mu}_n$ consists of i.i.d. centered random vectors. In addition,
%\begin{align*}
%&\Ep[(X_{i j}^*-\bar{X}_{j})^2|(X_i)_{i=1}^n]=\frac{1}{n}\sum_{i=1}^n X_{i j}^2-\bar{X}_{j}^2,\\
%&\Ep[|X_{i j}^*-\bar{X}_{j}|^3|(X_i)_{i=1}^n]=\frac{1}{n}\sum_{i=1}^n |X_{i j}-\bar{X}_{j}|^3\leq \hat{L}_n,\\
%&\Ep\left[\max_{1\leq j\leq p}|X_{i j}^*-\bar{X}_{j}|^31\left\{\max_{1\leq j\leq p}|X_{i j}^*-\bar{X}_{j}|^3>\frac{\sqrt{n}}{4\psi\log p}\right\}|(X_i)_{i=1}^n\right]\leq \hat{M}_{n,X}.
%\end{align*}
%Also, 
\qed

\subsection*{Proof of Proposition \ref{cor: EB corollary}}
Here $c,C$ are constants depending only on $b$ and $K$ under (E.1), and on $b,q$, and $K$ under (E.2); their values may change from place to place. 
% In this proof, $c$ and $C$ are constants that depend only on $b$ and $K$ under (E.1) and on $b$, $q$, and $K$ under (E.2) but their values may change at each appearance.
% Also, for brevity of notation, in this proof we implicitly assume that $i$ is varying over $\{1,\dots,n\}$ and $j$ and $k$ are varying over $\{1,\dots,p\}$.
We first note that, for sufficiently small $c>0$, we may assume that  
\begin{equation}
\label{eq: trivial bound}
B_n^2\log^{7}(p n) \leq cn, 
\end{equation}
since otherwise we can make the assertion of the lemma trivial by setting $C$ sufficiently large. 
%Moreover, by the same argument as that used in the proof of Proposition \ref{cor: MB corollary}, the problem reduces to the case of rectangles $\mA=\mA^{\re}$; that is, it suffices to prove the bounds (\ref{eq: EB example E1}) and (\ref{eq: EB example E2}) with $\rho_n^{EB}(\mA^{\si})$ replaced by $\rho_n^{EB}(\mA^{\re})$ and condition (M.1$'$) replaced by (M.1).
%For the latter problem, 
To prove the proposition, we will apply Theorem \ref{thm: empirical bootstrap CLT} separately under (E.1) and under (E.2).

\medskip
\textbf{Case  (E.1)}.
With (\ref{eq: trivial bound}) in mind, by the proof of Proposition \ref{cor: MB corollary}, we see that $\Pr (\Delta_{n,r} > b/2) \leq \alpha/6$, so that with probability larger than $1-\alpha/6$, 
$b/2 \leq \En[ (X_{ij}-\bar{X}_{j})^{2}] \leq C B_{n}$ for all $j=1,\dots,p$. 
We turn to bounding  $\hat{L}_n$. Using the inequality $| a -b |^{3} \leq 4(|a|^{3}+|b|^{3})$ together with Jensen's inequality, we have 
\[
\hat{L}_n \leq 4 (\max_{1\leq j\leq p}\En[|X_{ij}|^3]+\max_{1\leq j\leq p}|\bar{X}_{j}|^3 ) \leq 8 \max_{1\leq j\leq p}\En[|X_{ij}|^3].
\]
By Lemma \ref{lem: maximal ineq nonnegative}, 
\begin{align*}
\Ep[ \max_{1\leq j\leq p}\En[|X_{ij}|^3] ] &\leq C \{ L_{n} + n^{-1} \Ep[\max_{1 \leq i \leq n} \max_{1 \leq j \leq p} |X_{ij}|^{3}] \log p\} \\
&\leq C \{ B_{n} + n^{-1} B_{n}^{3} \log^{4} (pn) \}.
\end{align*}
Note that $\| |X_{ij}|^{3} \|_{\psi_{1/3}} \leq \| X_{ij} \|_{\psi_{1}}^{3} \leq B_{n}^{3}$, so that applying  Lemma \ref{lem: deviation ineq nonnegative} (i) with $\beta = 1/3$, we have for every $t > 0$, 
\[
\Pr ( \hat{L}_n \geq C \{ B_{n} + n^{-1} B_{n}^{3} \log^{4} (pn)  + n^{-1} B_{n}^{3} t^{3} \} ) \leq 3e^{-t}. 
\]
Taking $t = \log (18/\alpha) \leq C \log (pn)$, we conclude that, with  $\overline{L}_{n}=C  B_{n}$ (recall (\ref{eq: trivial bound})), $\Pr ( \hat{L}_n > \overline{L}_{n} ) \leq \alpha/6$.

Next, consider $\hat{M}_{n,X}(\phi_n)$. Observe that 
\[
\max_{1 \leq j \leq p}|X_{ij} - \bar{X}_{j}| \leq 2 \max_{1 \leq i \leq n} \max_{1 \leq j \leq p}|X_{ij}|,
\]
so that
\[
\Pr ( \hat{M}_{n,X}(\phi_n) > 0 ) \leq \Pr ( \max_{i,j} |X_{ij}| > \sqrt{n}/(8\phi_n \log p) ).
\]
Since $\| X_{ij} \|_{\psi_{1}} \leq B_{n}$, the right-hand side is bounded by 
\[
2(pn) \exp \{ - \sqrt{n}/(8B_{n}\phi_n \log p) \}. 
\]
Observe that 
\[
B_{n}\phi_n \log p \leq Cn^{1/6} B_{n}^{2/3} \log^{1/3} (p n), 
\]
so that using (\ref{eq: trivial bound}) with $c$ being sufficiently small, we conclude that 
\begin{align*}
\Pr ( \hat{M}_{n,X}(\phi_n) > 0 ) 
&\leq 2(p n) \exp\left(-\frac{n^{1/3}}{8C B_n^{2/3} \log^{1/3}(p n)}\right)\\
&\leq 2(p n) \exp\left(-\frac{n^{1/3}}{8C B_n^{2/3} \log^{7/3}(p n)}\cdot \log^2 (p n)\right)\\
&\leq 2(p n) \exp\left(- \frac{\log(p n) \log(1/\alpha)}{8 c^{1/3} C K}\right)\leq \alpha/6.
\end{align*}

To bound $\hat{M}_{n,Y}(\phi_n)$, observe that conditional on $X_{1},\dots,X_{n}$, $\| S_{nj}^{eX} \|_{\psi_{2}} \leq CB_{n}^{1/2}$ for all $j=1,\dots,p$ on the event $\max_{1\leq j\leq p}\En[(X_{ij}-\bar{X}_{j})^{2}] \leq CB_{n}$, which holds with probability larger than $1-\alpha/6$.
Hence, employing the same argument as that used to bound $\hat{M}_{n,X}(\phi_{n})$, we conclude that
\[
\Pr ( \hat{M}_{n,Y}(\phi_n) > 0 ) \leq \alpha/6 + \alpha/6 = \alpha/3, 
\]
which implies that 
\[
\Pr ( \hat{M}_{n}(\phi_{n}) = 0 ) > 1-(\alpha/6+\alpha/3) = 1-\alpha/2. 
\]

Taking these together, by Theorem \ref{thm: empirical bootstrap CLT}, with probability larger than $1-(\alpha/6+\alpha/6+\alpha/2)=1-5\alpha/6$, we have 
\[
\rho_{n}^{EB}(\mA^{\re}) \leq \rho_{n}^{MB}(\mA^{\re}) + C \{ n^{-1} B_{n}^{2} \log^{7} (pn) \}^{1/6}.
\]
The final conclusion follows from Proposition \ref{cor: MB corollary}.

\medskip
\textbf{Case  (E.2)}.
In this case, in addition to (\ref{eq: trivial bound}), we may assume that 
\begin{equation}\label{eq: trivial bound 3}
\frac{B_n^2\log^{3}(pn)}{\alpha^{2/q}n^{1-2/q}}\leq c \leq 1
\end{equation}
for sufficiently small $c>0$, since otherwise the assertion of the proposition is trivial by setting $C$ sufficiently large. Then as in the previous case,  by the proof of Proposition \ref{cor: MB corollary},
with probability larger than $1-\alpha/6$, $b/2 \leq \En[(X_{ij}-\bar{X}_{j})^2] \leq CB_{n}$ for all $j=1,\dots,p$.

To bound $\hat{L}_{n}$, recall that $\hat{L}_{n} \leq 8 \max_{1 \leq j \leq p} \En[|X_{ij}|^{3}]$, and by Lemma \ref{lem: maximal ineq nonnegative}, 
\[
\Ep[ \max_{1 \leq j \leq p} \En[|X_{ij}|^{3}] ] \leq C(B_{n} + B_{n}^{3} n^{-1+3/q} \log p). 
\]
Hence by applying Lemma \ref{lem: deviation ineq nonnegative} (ii) with $s=q/3$, we have for every $t > 0$, 
\[
\Pr ( \hat{L}_{n} \geq C(B_{n} + B_{n}^{3} n^{-1+3/q} \log p) + n^{-1} t ) \leq C t^{-q/3} \Ep[\max_{i,j} |X_{ij}|^{q}] \leq Ct^{-q/3} n B_{n}^{q}. 
\]
Solving $Ct^{-q/3} n B_{n}^{q} = \alpha/6$, we conclude that $\Pr ( \hat{L}_{n} \geq \overline{L}_{n} ) \leq \alpha/6$
where $\overline{L}_{n}= C(B_{n} + B_{n}^{3} n^{-1+3/q}\alpha^{-3/q} \log p)$.

Next, consider $\hat{M}_{n,X}(\phi_{n})$. As in the previous case, 
\[
\Pr ( \hat{M}_{n,X}(\phi_n) > 0 ) \leq \Pr ( \max_{i,j} |X_{ij}| > \sqrt{n}/(8\phi_{n} \log p) ).
\]
Since the right-hand side is nondecreasing in $\phi_n$, and 
\[
\phi_n \leq c B_{n}^{-1} n^{1/2 - 1/q} \alpha^{1/q} (\log p)^{-1},
\]
we have (by choosing the constant $C$ in $\overline{L}_{n}$ large enough)
\begin{align*}
&\Pr ( \max_{i,j} |X_{ij}| > \sqrt{n}/(8\phi_{n} \log p) ) \\
&\quad \leq n \max_{i} \Pr ( \max_{j} |X_{ij}| > C B_{n}n^{1/q} \alpha^{-1/q} ) \leq \alpha/6.
\end{align*}
For $\hat{M}_{n,Y}(\phi_{n})$, we make use of the argument in the previous case, and conclude that 
\[
\Pr ( \hat{M}_{n,Y}(\phi_{n}) > 0 ) \leq \alpha/2.
\]
The rest of the proof is the same as in the previous case. Note that 
\[
\left ( \frac{\overline{L}_{n}^{2} \log^{7}(pn)}{n} \right )^{1/6} \leq C \left [ \left ( \frac{B_n^{2} \log^{7} (pn)}{n} \right )^{1/6} + \left (\frac{B_n^{2} \log^{3} (pn)}{ \alpha^{2/q}n^{1-2/q}}\right )^{1/2} \right ],
\]
and because of (\ref{eq: trivial bound 3}), the second term inside the bracket on the right-hand side is at most 
\[
\left (\frac{B_n^{2} \log^{3}(pn)}{ \alpha^{2/q}n^{1-2/q}}\right )^{1/3}.
\]
This completes the proof in this case.
\qed

%\subsection*{Proof of Corollary \ref{cor: EB corollary log-concave distributions}}
%The proof is analogous to that of Corollary \ref{cor: MB corollary log-concave distributions}
%\qed


\begin{thebibliography}{99}
\bibitem[Adamczak(2008)]{A08a}
Adamczak, R. (2008). A tail inequality for suprema of unbounded empirical processes with applications to Markov chains. {\em Electron. J. Probab.} \textbf{13} 1000-1034.
\bibitem[Adamczak(2010)]{A08}
Adamczak, R. (2010). A few remarks on the operator norm of random Toeplitz matrices. {\em J. Theoret. Probab.} \textbf{23} 85-108.
\bibitem[Ball(1993)]{B93}
Ball, K. (1993). The reverse isoperimetric problem for Gaussian measure. {\em Discrete Comput. Geom.} \textbf{10} 411-420
\bibitem[Barvinok(2013)]{B13}
Barvinok, A. (2013). Thrifty approximations of convex bodies by polytopes. {\em Int. Math. Res. Not.} \textbf{00} 1-16.
\bibitem[Bentkus(1985)]{Bentkus85}
Bentkus, V. (1985). Lower bounds for the rate of convergence in the central limit theorem in Banach spaces. {\em Lithuanian Mathematical Journal} \textbf{25} 312-320.
\bibitem[Bentkus(1986)]{B86}
Bentkus, V. (1986). Dependence of the Berry-Esseen estimate on the dimension (in Russian). {\em Litovsk. Mat. Sb.} \textbf{26} 205-210.
\bibitem[Bentkus(2003)]{Bentkus03}
Bentkus, V. (2003). On the dependence of the Berry-Esseen bound on dimension. {\em J. Stat. Plann. Infer.} \textbf{113} 385-402.
\bibitem[Bhattacharya(1975)]{B75}
Bhattacharya, R. (1975). On the errors of normal approximation. {\em Ann. Probab.} \textbf{3} 815-828.
\bibitem[Bhattacharya and Rao(1986)]{BR86}
Bhattacharya, R. and Rao, R. (1986). {\em Normal approximation and asymptotic expansions}. Wiley.
\bibitem[Bolthausen(1984)]{Bolthausen84}
Bolthausen, E. (1984). An estimate of the remainder in a combinatorial central limit theorem. {\em Z. Wahrsch. Verw. Gebiete} \textbf{66} 379-386.
\bibitem[Borel(1974)]{B74}
Borell, C. (1974). Convex measures on locally convex spaces. {\em Ark. Mat.} \textbf{12} 239-252.
\bibitem[Boucheron et al.(2013)]{BLM2013}
Boucheron, S., Lugosi, G. and Massart, P. (2013). {\em Concentration Inequalities: A Nonasymptotic Theory of Independence}. Oxford University Press.
\bibitem[Chatterjee(2005a)]{Chatterjee2005a}
Chatterjee, S. (2005a). A simple invariance theorem. arXiv:math/0508213.
\bibitem[Chatterjee(2005b)]{Chatterjee2005b}
Chatterjee, S. (2006). A generalization of Lindeberg's principle. {\em Ann. Probab.} \textbf{34} 2061-2076.
\bibitem[Chatterjee and Meckes(2008)]{ChatterjeeMeckes2008}
Chatterjee, S. and Meckes, E. (2008). Multivariate normal approximation using exchangeable pairs. {\em ALEA Lat. Am. J. Probab. Math. Stat.} \textbf{4} 257-283.
\bibitem[Chen and Fang(2011)]{CF11}
Chen, L. and Fang, X. (2011). Multivariate normal approximation by Stein's method: the concentration inequality approach. arXiv:1111.4073.
\bibitem[Chernozhukov et al.(2013)]{CCK12a}
Chernozhukov, V., Chetverikov, D. and Kato, K. (2013). Gaussian approximations and multiplier bootstrap for maxima of sums of high-dimensional random vectors. {\em Ann. Statist.} \textbf{41} 2786-2819.
\bibitem[Chernozhukov et al.(2013)]{CCK12d}
Chernozhukov, V., Chetverikov, D. and Kato, K. (2013). Supplemental Material to ``Gaussian approximations and multiplier bootstrap for maxima of sums of high-dimensional random vectors''. {\em Ann. Statist.} \textbf{41} 2786-2819.
%\bibitem[Chernozhukov et al.(2013)]{CCK13}
%Chernozhukov, V., Chetverikov, D. and Kato, K. (2013). Testing many moment inequalities. arXiv:1312.7614.
\bibitem[Chernozhukov et al.(2014a)]{CCK12b}
Chernozhukov, V., Chetverikov, D. and Kato, K. (2014a). Gaussian approximation of suprema of empirical processes. {\em Ann. Statist.} \textbf{42} 1564-1597.
\bibitem[Chernozhukov et al.(2014b)]{CCK12c}
Chernozhukov, V., Chetverikov, D. and Kato, K. (2014b). Comparison and anti-concentration bounds for maxima of Gaussian random vectors. To appear in {\em Probab. Theory Related Fields}. 
% \bibitem[Gu\'{e}don et al.(2013)]{GNT13}
% Gu\'{e}don, O., Nayar, P., and Tkocz, T. (2013). Concentration inequalities and geometry of convex bodies. Lecture notes.
\bibitem[Dudley(1999)]{Dudley02}
Dudley, R.M. (1999). {\em Uniform Central Limit Theorems}. Cambridge University Press.
\bibitem[Einmahl and Li(2008)]{EL08}
Einmahl, U. and Li, D. (2008). Characterization of LIL behavior in Banach space. {\em Trans. Amer. Math. Soc.} \textbf{360} 6677-6693.
\bibitem[G\"{o}tze(1991)]{G91}
G\"{o}tze, F. (1991). On the rate of convergence in the multivariate CLT. {\em Ann. Probab.} \textbf{19} 724-739.
\bibitem[Goldstein and Rinott (1996)]{GoldsteinRinott1996}
Goldstein, L., and Rinott, Y. (1996) Multivariate normal approximations by Stein's method and size bias couplings. {\em J. Appl. Probab.}  \textbf{33} 1-17.
\bibitem[Klivans et al.(2008)]{KDS08}
Klivans, A., O'Donnell, R., and Servedio, R. (2008). Learning geometric concepts via Gaussian surface area. {\em 49th Annual IEEE Symposium on Foundations of Computer Science}.
\bibitem[Ledoux and Talagrand(1991)]{LT91}
Ledoux, M. and Talagrand, M. (1991). {\em Probability in Banach Spaces}. Springer. 
\bibitem[Massart(2000)]{M00}
Massart, P. (2000). About the constants in Talagrand's concentration inequalities for empirical processes. {\em Ann. Probab.} \textbf{28} 863-884.
\bibitem[Milman and Schechtman(1986)]{MS86}
Milman, V.D. and Shechtman, G. (1986). {\em Asymptotic Theory of Finite Dimensional Normed Spaces}.  Lecture Notes in Mathematics 1200 Springer. 
\bibitem[Nagaev(1976)]{N76}
Nagaev, S. (1976). An estimate of the remainder term in the multidimensional central limit theorem. {\em Proc. Third Japan-USSR Symp. Probab. Theory.} \textbf{550} 419-438.
\bibitem[Nazarov(2003)]{N03}
Nazarov, F. (2003). On the maximal perimeter of a convex set in $\R^n$ with respect to a Gaussian measure. In: {\em Geometric Aspects of Functional Analysis}, Lecture Notes in Mathematics Volume 1807, Springer, pp. 169-187.
\bibitem[Panchenko(2013)]{Panchenko2013}
Panchenko, D. (2013). {\em The Sherrington-Kirkpatrick Model}. Springer.
\bibitem[Praestgaard and Wellner(1993)]{PW93}
Praestgaard, J. and Wellner, J.A. (1993). Exchangeably weighted bootstraps of the general empirical processes. {\em Ann. Probab.} \textbf{21} 2053-2086.
\bibitem[Reinert and R\"{o}llin(2009)]{ReinertRollin2009}
Reinert, G. and R\"{o}llin, A. (2009). Multivariate normal approximation with Stein's method of exchangeable pairs under a general linearity condition. {\em Ann. Probab.} \textbf{37} 2150-2173.
\bibitem[R\"{o}llin(2011)]{Rollin2011}
R\"{o}llin, A. (2011). Stein's method in high dimensions with applications, {\em  Ann. Inst. H. Poincar\'{e} Probab. Statist.} \textbf{49}, 529-549.
\bibitem[Sazonov(1968)]{S68}
Sazonov, V. (1968). On the multi-dimensional central limit theorem. {\em Sanhya Ser. A} \textbf{30} 181-204.
\bibitem[Sazonov(1981)]{S81}
Sazonov, V. (1981). {\em Normal Approximations: Some Recent Advances}. Lecture Notes in Mathematics Volume 879, Springer. 
\bibitem[Senatov(1980)]{S80}
Senatov, V. (1980). Several estimates of the rate of convergence in the multidimensional CLT. {\em Dokl. Acad. Nauk SSSR} \textbf{254} 809-812.
\bibitem[Sweeting(1977)]{S77}
Sweeting, T. (1977). Speed of convergence for the multidimensional central limit theorem. {\em Ann. Probab.} \textbf{5} 28-41.
\bibitem[Slepian(1962)]{Slepian1962}
Slepian, D. (1962). The one-sided barrier problem for Gaussian noise. {\em Bell Syst. Tech. J.} \textbf{41} 463-501.
\bibitem[Stein(1981)]{Stein1981}
Stein, C. (1981). Estimation of the mean of a multivariate normal distribution. {\em Ann. Statist.} \textbf{9} 1135-1151.
\bibitem[Talagrand(2003)]{Talagrand2003}
Talagrand, M. (2003). {\em Spin Glasses: A Challenge for Mathematicians}.  Springer.
%\bibitem[Talagrand(1996)]{T96}
%Talagrand, M. (1996). New concentration inequalities in product spaces. {\em Invent. Math.} \textbf{126} 505-563.
%\bibitem[van der Geer(2000)]{Geer_book}
%van der Geer, S. (2000). {\em Empirical Processes in M-estimation}. Cambridge University Press. 
\bibitem[van der Vaart and Wellner(1996)]{VW96}
van der Vaart, A.W. and Wellner, J.A. (1996). {\em Weak Convergence and Empirical Processes: With Applications to Statistics}. Springer.
\end{thebibliography}
\end{document}